\documentclass[english]{article}

\usepackage[english]{babel}
\usepackage{latexsym}

\usepackage{amsmath}
\usepackage{amsfonts}
\usepackage{amssymb}
\usepackage{amsthm}
\usepackage{verbatim}
\usepackage{psfrag}
\usepackage{graphicx}

\theoremstyle{theorem}
\newtheorem{lemm}{Lemma}[section]
\newtheorem{prop}[lemm]{Proposition}
\newtheorem{coro}[lemm]{Corollary}
\newtheorem{theo}{Theorem}
\theoremstyle{remark}
\newtheorem{defi}[lemm]{Definition}
\newtheorem{nota}[lemm]{Notation}

\newtheorem{exem}[lemm]{Example}
\newtheorem{rema}[lemm]{Remark}

\usepackage{color}

\definecolor{grey}{rgb}{0.6,0.6,0.6}
\newcommand{\proofend}{\hfill $\square$}
\newcommand{\Z}[1]{\mathbb{Z}/#1\mathbb{Z}}

\newcommand{\defn}[1]{{\em #1}}
\newcommand{\CrP}{\mathrm{Bir}(\Pn)}

\newcommand{\PGLn}[1]{\PGL(#1,\K)}

\newcommand{\Ker}{\mathrm{Ker}}
\newcommand{\pr}{pr}
\newcommand{\Diag}[3]{[#1:#2:#3]}
\newcommand{\DiaG}[4]{[#1:#2:#3:#4]}
\newcommand{\PGL}{\mathrm{PGL}}
\newcommand{\GL}{\mathrm{GL}}
\newcommand{\im}{{\bf i}}
\newcommand{\ipi}{{\bf i}\pi}
\newcommand{\K}{\mathbb{C}}
\newcommand{\Bir}{\mathrm{Bir}}
\newcommand{\Fix}{\mathrm{Fix}}

\newcommand{\Aut}{\mathrm{Aut}}
\newcommand{\Sym}{\mathrm{Sym}}
\newcommand{\Alt}{\mathrm{Alt}}
\newcommand{\rkPic}[1]{\mathrm{rk\ Pic}(#1)}
\newcommand{\Pn}{\mathbb{P}^2}
\newcommand{\Pic}[1]{\mathrm{Pic}(#1)}
\newcommand{\drawat}[3]{\makebox[0pt][l]{\raisebox{#2}{\hspace*{#1}#3}}}

\newcommand{\h}[1]{\hspace{-#1mm}}

\newcommand{\nump}[1]{\ensuremath{\begin{array}{|c|}\hline {\mathrm{#1}} \\ \hline\end{array}}}
\newcommand{\Cs}[1]{\mathit{Cs}_{#1}}

\begin{document}
\title{Linearisation of finite Abelian subgroups of the Cremona group of the plane}
\author{J\'er\'emy Blanc}
\maketitle

\begin{abstract}
Given a finite Abelian subgroup of the Cremona group of the plane, we provide a way to decide whether it is birationally conjugate to a group of automorphisms of a minimal surface.

In particular, we prove that a finite cyclic group of birational transformations of the plane is linearisable if and only if none of its non-trivial elements fix a curve of positive genus. For finite Abelian groups, there exists only one surprising exception, a group isomorphic to $\Z{2}\times\Z{4}$, whose non-trivial elements do not fix a curve of positive genus but which is not conjugate to a group of automorphisms of a minimal rational surface.

We also give some descriptions of automorphisms (not necessarily of finite order) of del Pezzo surfaces and conic bundles.\end{abstract}

\section{Introduction}
\subsection{The main questions and results}
In this paper, every surface will be \emph{complex, rational, algebraic} and \emph{smooth}, and except for $\K^2$, will also be projective. By an \emph{automorphism} of a surface we mean a biregular algebraic morphism from the surface to itself. The group of automorphisms (respectively of birational transformations) of a surface $S$ will be denoted by $\Aut(S)$ (respectively by $\Bir(S)$).

The group $\Bir(\Pn)$ is classically called the \defn{Cremona group}. 
Taking some surface $S$, any birational map $S\dasharrow \Pn$ conjugates $\Bir(S)$ to $\Bir(\Pn)$; any subgroup of $\Bir(S)$ may therefore be viewed as a subgroup of the Cremona group, up to conjugacy.

The minimal surfaces are $\Pn$, $\mathbb{P}^1\times\mathbb{P}^1$ and the Hirzebruch surfaces $\mathbb{F}_n$ for $n\geq 2$; their groups of automorphisms are a classical object of study, and their structures are well known (see for example \cite{bib:BeaBook}). These groups are in fact the maximal connected algebraic subgroups of the Cremona group (see \cite{bib:MuUm}, \cite{bib:Um}).

Given some group acting \emph{birationally} on a surface, we would like to determine some geometric properties that allow us to decide whether the group is conjugate to a group of \emph{automorphisms} of a \emph{minimal} surface, or equivalently to decide whether it belongs to a maximal connected algebraic subgroup of the Cremona group. This conjugation looks like a linearisation, as we will see below, and explains our title.

We observe that the set of points of a minimal surface which are fixed by a non-trivial automorphism is the union of a finite number of points and rational curves. Given a group $G$ of birational transformations of a surface, the following properties are thus related (note that for us the genus is the geometric genus, so that a curve has positive genus if and only if it is not rational); property $(F)$ is our candidate for the geometric property for which we require:\begin{center}\begin{tabular}{rl}
$(F)$ & \begin{tabular}[t]{l}No non-trivial element of $G$ fixes (pointwise) a curve of positive genus.\end{tabular}\\
$(M)$ &\begin{tabular}[t]{l}The group $G$ is birationally conjugate to a group of automorphisms of\\ a minimal surface.\end{tabular}
\end{tabular}\end{center}

The fact that a curve of positive genus is not collapsed by a birational transformation of surfaces implies that property $(F)$ is a conjugacy invariant; it is clear that the same is true of property $(M)$. The above discussion implies that $(M)\Rightarrow (F)$; we would like to prove the converse.

The implication $(F)\Rightarrow (M)$ is true for finite cyclic groups of prime order (see \cite{bib:BeB}). The present article describes precisely the case of finite Abelian groups. 
We prove that $(F) \Rightarrow (M)$ is true for finite cyclic groups of any order, and that we may restrict the minimal surfaces to $\Pn$ or $\mathbb{P}^1\times\mathbb{P}^1$. In the case of finite Abelian groups, there exists, up to conjugation, only one counterexample to the implication, which is represented by a group isomorphic to $\Z{2}\times\Z{4}$ acting biregularly on a special conic bundle. Precisely, we will prove the following results, announced without proof as Theorems~4.4 and~4.5 in \cite{bib:JBCR}:
\begin{theo}
\label{Thm:Cyclic}
Let $G$ be a finite {\bf cyclic} subgroup of order~$n$ of the Cremona group. The following conditions are equivalent:
\begin{itemize}
\item
If $g \in G$, $g\not=1$, then~$g$ does not fix a curve of positive genus.
\item
$G$ is birationally conjugate to a subgroup of $\Aut(\Pn)$.
\item
$G$ is birationally conjugate to a subgroup of $\Aut(\mathbb{P}^1\times\mathbb{P}^1)$.
\item
$G$ is birationally conjugate to the group of automorphisms of $\Pn$ generated by $(x:y:z) \mapsto (x:y:e^{2\ipi/n} z)$.
\end{itemize}
\end{theo}
\begin{theo}
\label{Thm:NonCyclic}
Let $G$ be a finite {\bf Abelian} subgroup of the Cremona group. The following conditions are equivalent:
\begin{itemize}
\item
If $g \in G$, $g\not=1$, then~$g$ does not fix a curve of positive genus.
\item
$G$ is birationally conjugate to a subgroup of $\Aut(\Pn)$, or to a subgroup of $\Aut(\mathbb{P}^1\times\mathbb{P}^1)$ or to the group $\Cs{24}$ isomorphic to $ \Z{2}\times\Z{4}$, generated by the two elements \begin{center}$\begin{array}{lll}(x:y:z)&\dasharrow&(yz:xy:-xz),\\
(x:y:z)&\dasharrow &( yz(y-z):xz(y+z):xy(y+z)).\end{array}$\end{center}
\end{itemize}
Moreover, this last group is conjugate neither to a subgroup of $\Aut(\Pn)$, nor to a subgroup of $\Aut(\mathbb{P}^1\times\mathbb{P}^1)$.
\end{theo}
Then, we discuss the case in which the group is infinite, respectively non-Abelian (Section~\ref{Sec:Counter}) and provide many examples of groups satisfying $(F)$ but not $(M)$.

Note that many finite groups which contain elements that fix a non-rational curve are known, see for example \cite{bib:Wim} or more recently \cite{bib:JBTh} and \cite{bib:Dol}. This can also occur if the group is infinite, see \cite{bib:BPV} and \cite{bib:JBCubIn}. In fact, the set of non-rational curves fixed by the elements of a group is a conjugacy invariant very useful in describing conjugacy classes (see \cite{bib:BaB}, \cite{bib:deF}, \cite{bib:JBSMF}).
\subsection{How to decide}
Given a finite Abelian group of birational transformations of a (rational) surface, we thus have a good way to determine whether the group is birationally conjugate to a group of automorphisms of a minimal surface (in fact to $\Pn$ or $\mathbb{P}^1\times\mathbb{P}^1$). If some non-trivial element fixes a curve of positive genus (i.e.\ if condition $(F)$ is not satisfied), this is false. Otherwise, if the group is not isomorphic to $\Z{2}\times\Z{4}$, it is birationally conjugate to a subgroup of $\Aut(\Pn)$ or of $\Aut(\mathbb{P}^1\times\mathbb{P}^1)$. There are exactly four conjugacy classes of groups isomorphic to $\Z{2}\times\Z{4}$ satisfying condition $(F)$ (see Theorem~\ref{Prp:TheClassification}); three are conjugate to a subgroup of $\Aut(\Pn)$ or $\Aut(\mathbb{P}^1\times\mathbb{P}^1)$, and the fourth (the group $\Cs{24}$ of Theorem~\ref{Thm:NonCyclic}, described in detail in Section~\ref{Sec:ExampleCs24}) is not.
\subsection{Linearisation of birational actions}
Our question is related to that of linearisation of birational actions on~$\K^2$. This latter question has been studied intensively for \emph{holomorphic} or \emph{polynomial} actions, see for example \cite{bib:DeKu}, \cite{bib:Kra} and \cite{bib:vdE}. Taking some group acting birationally on $\K^2$, we would like to know if we may birationally conjugate this action to have a linear action. Note that working on $\Pn$ or $\K^2$ is the same for this question. Theorem~\ref{Thm:Cyclic} implies that for finite cyclic groups, being linearisable is equivalent to fulfilling condition $(F)$. This is not true for finite Abelian groups in general, since some groups acting biregularly on $\mathbb{P}^1\times\mathbb{P}^1$ are not birationally conjugate to groups of automorphisms of $\Pn$.
Note that Theorem~\ref{Thm:Cyclic} implies the following result on linearisation, also announced in \cite{bib:JBCR} (as Theorem 4.2):
\begin{theo}
\label{Thm:Roots}
Any birational map which is a root of a non-trivial linear automorphism of finite order of the plane is conjugate to a linear automorphism of the plane.\end{theo}

\subsection{The approach and other results}
Our approach -- followed in all the modern articles on the subject -- is to view the finite subgroups of the Cremona group as groups of (biregular) automorphisms of smooth projective rational surfaces and then to assume that the action is minimal (i.e.\ that it is not possible to blow-down some curves and obtain once again a biregular action on a smooth surface). Manin and Iskovskikh (\cite{bib:Man} and \cite{bib:Isk3}) proved that the only possible cases are action on del Pezzo surfaces or conic bundles. We will clarify this classification, for finite Abelian groups fillfulling (F), by proving the following result:
\begin{theo}\label{Thm:Classifmin}
Let $S$ be some smooth projective rational surface and let $G\subset \Aut(S)$ be a finite Abelian group of automorphisms of $S$ such that
\begin{itemize}
\item
the pair $(G,S)$ is minimal;
\item
if $g \in G$, $g\not=1$, then~$g$ does not fix a curve of positive genus.
\end{itemize}
Then, one of the following occurs:
\begin{enumerate}
\item[\upshape 1.]
The surface $S$ is minimal, i.e.\ $S\cong \mathbb{P}^2$, or $S\cong \mathbb{F}_n$ for some integer $n\not=1$.
\item[\upshape 2.]
The surface $S$ is a del Pezzo surface of degree $5$ and $G\cong \Z{5}$.
\item[\upshape 3.]
The surface $S$ is a del Pezzo surface of degree $6$ and $G\cong \Z{6}$.
\item[\upshape 4.]
The pair $(G,S)$ is isomorphic to the pair $(\Cs{24},\hat{S_4})$ defined in Section~{\upshape \ref{Sec:ExampleCs24}}.
\end{enumerate}
\end{theo}
We will then prove that all the pairs in cases $1,2$ and $3$ are birationally equivalent to a group of automorphisms of $\mathbb{P}^1\times\mathbb{P}^1$ or $\mathbb{P}^2$, and that this is not true for case $4$. In fact, we are able to provide the precise description of all conjugacy classes of finite Abelian subgroups of $\Bir(\Pn)$ satisfying $(F)$:

\begin{theo}
\label{Prp:TheClassification}
Let $G$ be a finite Abelian subgroup of the Cremona group such that no non-trivial element of $G$ fixes a curve of positive genus. Then, $G$ is birationally conjugate to one and only one of the following:
\begin{flushleft}\begin{tabular}{llll}
{\upshape [1]} & $G \cong \Z{n}\times\Z{m}$& g.b.&$(x,y)Ê\mapsto (\zeta_n x,y)$ and
$(x,y) \mapsto (x,\zeta_m y)$\\
{\upshape [2]} & $G \cong \Z{2}\times\Z{2n} $& g.b.&$(x,y)Ê\mapsto (x^{-1},y)$ and
$(x,y) \mapsto (-x,\zeta_{2n} y)$\\
{\upshape [3]} & $G \cong (\Z{2})^2\times\Z{2n} $& g.b.&$(x,y)Ê\mapsto (\pm x^{\pm 1},y)$ and
$(x,y) \mapsto (x,\zeta_{2n} y)$\\
{\upshape [4]} & $G \cong (\Z{2})^3$& g.b. &$(x,y)Ê\mapsto (\pm x,\pm y)$ and
$(x,y) \mapsto (x^{-1},y^{-1})$\\
{\upshape [5]} & $G \cong (\Z{2})^4$& g.b. &$(x,y)Ê\mapsto (\pm x^{\pm 1},\pm y^{\pm 1})$\\
{\upshape [6]} & $G \cong \Z{2}\times \Z{4}$& g.b. & $(x,y)Ê\mapsto (x^{-1},y^{-1})$ and $(x,y) \mapsto (-y,x)$\\
{\upshape [7]} & $G \cong (\Z{2})^3$& g.b. & $(x,y)Ê\mapsto (-x,-y)$, $(x,y)Ê\mapsto (x^{-1},y^{-1})$, 
\\
& & & and $(x,y) \mapsto (y,x)$ \\
{\upshape [8]} & $G \cong (\Z{2})\times(\Z{4})$& g.b. & $(x:y:z)\dasharrow ( yz(y-z):xz(y+z):xy(y+z))$ \\
& & & and $(x:y:z)\dasharrow (yz:xy:-xz)$ \\
{\upshape [9]} & $G \cong (\Z{3})^2$& g.b. & $(x:y:z)\mapsto (x:\zeta_3 y:(\zeta_3)^2 z)$
\\
& & & and $(x:y:z) \mapsto (y:z:x)$ \\
\end{tabular}\\
{\upshape (}where $n,m$ are positive integers,~$n$ divides~$m$ and $\zeta_n=e^{2\ipi/n}${\upshape )}.\end{flushleft}
Furthermore, the groups in cases {\upshape [1]} through {\upshape [7]} are birationally conjugate to subgroups of $\Aut(\mathbb{P}^1\times\mathbb{P}^1)$, but the others are not. The groups in cases {\upshape [1]} and {\upshape [9]} are birationally conjugate to subgroups of $\Aut(\Pn)$, but the others are not.
\end{theo}

\bigskip

To prove these results, we will need a number of geometric results on automorphisms of rational surfaces, and in particular on automorphisms of conic bundles and del Pezzo surfaces (Sections~\ref{Aut:ConBundles} to~\ref{Sec:DelPezzo}). We give for example the classification of all the twisting elements (that exchange the two components of a singular fibre) acting on conic bundles in Proposition~\ref{Prp:DescriptionTwistingElementsFinite} (for the elements of finite order) and Proposition~\ref{Prp:DescriptionTwistingElementsInfinite} (for those of infinite order); these are the most important elements in this context (see Lemma~\ref{Lem:MinTripl}). 
We also prove that actions of (possibly infinite) Abelian groups on del Pezzo surfaces satifying $(F)$ are minimal only if the degree is at least $5$ (Section~\ref{Sec:DelPezzo}) and describe these cases precisely (Sections~\ref{Sec:DelPezzo6}, \ref{Sec:DelPezzo5} and~\ref{Sec:DelPezzo}). We also show that a finite Abelian group acting on a projective smooth surface $S$ such that $(K_S)^2\geq 5$ is birationally conjugate to a group of automorphisms of $\mathbb{P}^1\times\mathbb{P}^1$ or $\Pn$ (Corollary~\ref{Coro:K2leq5}) and in particular satisfies $(F)$.
\subsection{Comparison with other work}
Many authors have considered the finite subgroups of $\Bir(\Pn)$. Among them, S.~Kantor \cite{bib:SK} gave a classification of the finite subgroups, which was incomplete and included some mistakes; A.~Wiman \cite{bib:Wim} and then I.V. Dolgachev and V.A. Iskovskikh \cite{bib:Dol} successively improved Kantor's results. The long paper \cite{bib:Dol} expounds the general theory of finite subgroups of $\Bir(\Pn)$ according to the modern techniques of algebraic geometry, and will be for years to come the reference on the subject. Our viewpoint and aim differ from those of \cite{bib:Dol}: we are only interested in Abelian groups in relation with the above conditions (F) and (M); this gives a restricted setting in which the theoretical approach is simplified and the results obtained are more accurate. In the study of del Pezzo surfaces, using the classification \cite{bib:Dol} of subgroups of automorphisms would require the examination of many cases; for the sake of readibility we prefered a direct proof. The two main theorems of \cite{bib:Dol} on automorphism of conic bundles (Proposition 5.3 and Theorem~5.7(2)) do not exclude groups satisfying property $(F)$ and do not give explicit forms for the generators of the groups or the surfaces.

\subsection{Aknowledgements}
This article is part of my PhD thesis \cite{bib:JBTh}; I am grateful to my advisor T.~Vust for his invaluable help during these years, to I. Dolgachev for helpful discussions, and thank J.-P.~Serre and the referees for their useful remarks on this paper. 

\section{Automorphisms of $\mathbb{P}^2$ or $\mathbb{P}^1\times\mathbb{P}^1$} \label{Sec:P1P2C2}
Note that a linear automorphism of $\mathbb{C}^2$ may be extended to an automorphism of either $\mathbb{P}^2$ or $\mathbb{P}^1\times\mathbb{P}^1$. Moreover, the automorphisms of finite order of these three surfaces are birationally conjugate. For finite Abelian groups, the situation is quite different. We give here the birational equivalence of these groups.

\begin{nota}
The element $\Diag{a}{b}{c}$ denotes the diagonal automorphism $(x:y:z) \mapsto (ax:by:cz)$ of $\Pn$, and $\zeta_m=e^{2\ipi/m}$.\end{nota}

\begin{prop}[Finite Abelian subgroups of $\Aut(\Pn)$]
\label{Prp:PGL3Cr}
Every finite Abelian subgroup of\hspace{0.2cm}$\Aut(\Pn)=\PGLn{3}$ is conjugate, in the Cremona group $\CrP$, to one and only one of the following:
\begin{itemize}
\item[\upshape{ 1.}]
A diagonal group, isomorphic to $\Z{n} \times \Z{m}$, where~$n$ divides~$m$, generated by $\Diag{1}{\zeta_n}{1}$ and $\Diag{\zeta_m}{1}{1}$. (The case $n=1$ gives the cyclic groups).
\item[\upshape{ 2.}]
 The special group $V_9$, isomorphic to $\Z{3} \times \Z{3}$, generated by $\Diag{1}{\zeta_3}{(\zeta_3)^2}$\\ and $(x:y:z) \mapsto (y:z:x)$.
\end{itemize}
Thus, except for the group $V_9$, two isomorphic finite Abelian subgroups of $\PGLn{3}$ are conjugate in $\CrP$.
\end{prop}
\begin{proof}
First of all, a simple calculation shows that every finite Abelian subgroup of $\PGLn{3}$ is either diagonalisable or conjugate to the group $V_9$.
Furthermore, since this last group does not fix any point, it is not diagonalisable, even in $\CrP$ \cite[Proposition A.2]{bib:KoS}.

Let $\mathcal{T}$ denote the torus of $\PGLn{3}$ constituted by diagonal automorphisms of $\Pn$. Let $G$ be a finite subgroup of $\mathcal{T}$; as an abstract group it is isomorphic to $\Z{n} \times \Z{m}$, where~$n$ divides~$m$. Now we can conjugate $G$ by a birational map of the form $h:(x,y) \dasharrow (x^ay^b,x^cy^d)$ so that it contains $\Diag{\zeta_m}{1}{1}$ (see \cite{bib:BeB} and \cite{bib:BlaMM}). Since~$h$ normalizes the torus $\mathcal{T}$, the group $G$ remains diagonal and contains the $n$-torsion of $\mathcal{T}$, hence it contains $\Diag{1}{\zeta_n}{1}$.
\end{proof}
\begin{coro}
Every finite Abelian group of linear automorphisms of $\K^2$ is birationally conjugate to a diagonal group, isomorphic to $\Z{n} \times \Z{m}$, where~$n$ divides~$m$, generated by $(x,y)\mapsto (\zeta_n x,y)$ and $(x,y)\mapsto (x,\zeta_m y)$.
\end{coro}
\begin{proof}
This follows from the fact that the group $\GL(2,\K)$ of linear automorphisms of $\K^2$ extends to a group of automorphisms of $\Pn$ that leaves the line at infinity invariant and fixes one point.\end{proof}

\begin{exem}
\label{Exa:TorP1}
Note that $\Aut(\mathbb{P}^1 \times \mathbb{P}^1)$ contains the group $(\K^{*})^2 \rtimes \Z{2}$, where $(\K^{*})^2$ is the group of automorphisms of the form $(x,y)\mapsto (\alpha x,\beta y)$, $\alpha,\beta \in \K^{*}$, and $\Z{2}$ is generated by the automorphism $(x,y) \mapsto (y,x)$.

The birational map $(x,y) \dasharrow (x:y:1)$ from $\mathbb{P}^1\times \mathbb{P}^1$ to $\Pn$ conjugates $(\K^{*})^2 \rtimes \Z{2}$ to the group of automorphisms of\hspace{0.2 cm}$\Pn$ generated by $(x:y:z)\mapsto (\alpha x:\beta y:z)$, $\alpha,\beta \in \K^{*}$ and $(x:y:z) \mapsto (y:x:z)$.
\end{exem}

\begin{prop}[Finite Abelian subgroups of $\Aut(\mathbb{P}^1\times\mathbb{P}^1)$] \label{Prp:AutP1P1B}
Up to birational conjugation, every finite Abelian subgroup of $\Aut(\mathbb{P}^1\times\mathbb{P}^1)$ is conjugate to one and only one of the following:
\begin{flushleft}\begin{tabular}{llll}
{\upshape [1]} & $G \cong \Z{n}\times\Z{m}$& g.b.&$(x,y)Ê\mapsto (\zeta_n x,y)$ and
$(x,y) \mapsto (x,\zeta_m y)$\\
{\upshape [2]} & $G \cong \Z{2}\times\Z{2n} $& g.b.&$(x,y)Ê\mapsto (x^{-1},y)$ and
$(x,y) \mapsto (-x,\zeta_{2n} y)$\\
{\upshape [3]} & $G \cong (\Z{2})^2\times\Z{2n} $& g.b.&$(x,y)Ê\mapsto (\pm x^{\pm 1},y)$ and
$(x,y) \mapsto (x,\zeta_{2n} y)$\\
{\upshape [4]} & $G \cong (\Z{2})^3$& g.b. &$(x,y)Ê\mapsto (\pm x,\pm y)$ and
$(x,y) \mapsto (x^{-1},y^{-1})$\\
{\upshape [5]} & $G \cong (\Z{2})^4$& g.b. &$(x,y)Ê\mapsto (\pm x^{\pm 1},\pm y^{\pm 1})$\\
{\upshape [6]} & $G \cong \Z{2}\times \Z{4}$& g.b. & $(x,y)Ê\mapsto (x^{-1},y^{-1})$ and $(x,y) \mapsto (-y,x)$\\
{\upshape [7]} & $G \cong (\Z{2})^3$& g.b. & $(x,y)Ê\mapsto (-x,-y)$, $(x,y)Ê\mapsto (x^{-1},y^{-1})$, 
\\
& & & and $(x,y) \mapsto (y,x)$ \\
\end{tabular}\\
{\upshape (}where $n,m$ are positive integers,~$n$ divides~$m$ and $\zeta_n=e^{2\ipi/n}${\upshape )}.\end{flushleft}
 Furthermore, the groups in $[1]$ are conjugate to subgroups of $\Aut(\Pn)$, but the others are not. 
\end{prop}
\begin{proof}
Recall that $\Aut(\mathbb{P}^1\times\mathbb{P}^1)=(\PGLn{2} \times \PGLn{2})\rtimes\Z{2}$. Let $G$ be some finite Abelian subgroup of $\Aut(\mathbb{P}^1\times\mathbb{P}^1)$; we now prove that $G$ is conjugate to one of the groups in cases $[1]$ through $[7]$.

First of all, if $G$ is a subgroup of the group $(\K^{*})^2\rtimes\Z{2}$ given in Example~\ref{Exa:TorP1}, then it is conjugate to a subgroup of $\Aut(\Pn)$ and hence to a group in case $[1]$.

Assume that $G \subset \PGLn{2} \times \PGLn{2}$ and denote by $\pi_1$ and $\pi_2$ the projections $\pi_i: \PGLn{2} \times \PGLn{2} \rightarrow \PGLn{2}$ on the $i$-th factor. Since $\pi_1(G)$ and $\pi_2(G)$ are finite Abelian subgroups of $\PGLn{2}$ each is conjugate to a diagonal cyclic group or to the group $x\dasharrow \pm x^{\pm 1}$, isomorphic to $(\Z{2})^2$. We enumerate the possible cases.

{\it If both groups $\pi_1(G)$ and $\pi_2(G)$ are cyclic,}
the group $G$ is conjugate to a subgroup of the diagonal torus $(\K^{*})^2$ of automorphisms of the form $(x,y)\mapsto (\alpha x,\beta y)$, $\alpha,\beta \in \K^{*}$.

{\it If exactly one of the two groups $\pi_1(G)$ and $\pi_2(G)$ is cyclic} we may assume,
up to conjugation in $\Aut(\mathbb{P}^1 \times \mathbb{P}^1)$, that $\pi_2(G)$ is cyclic, generated by $y \mapsto \zeta_m y$, for some integer $m\geq 1$, and that $\pi_1(G)$ is the group $x\dasharrow \pm x^{\pm 1}$. We use the exact sequence $1\rightarrow G\cap \ker \pi_2\rightarrow G \rightarrow \pi_2(G)\rightarrow 1$ and find, up to conjugation, two possibilities for $G$:
\begin{center}$
\begin{array}{lllll}
\mbox{(a)}&\mbox{$G$ is generated by}&(x,y)\mapsto (x^{-1},y)&\mbox{and}&(x,y) \mapsto (-x,\zeta_{m}y).\\
\mbox{(b)}&\mbox{$G$ is generated by}& (x,y) \mapsto (\pm x^{\pm 1},y)&\mbox{and}&(x,y)\mapsto (x,\zeta_m y).\end{array}
$\end{center}
If~$m$ is even, we obtain respectively [2] and [3] for $n=m/2$. If~$m$ is odd, the two groups are equal; conjugating by $\varphi:(x,y) \dasharrow (x,y(x+x^{-1}))$ (which conjugates $(x,y)\mapsto (-x,y)$ to $(x,y)\mapsto (-x,-y)$) we obtain the group [2] for $n=m$.

{\it If both groups $\pi_1(G)$ and $\pi_2(G)$ are isomorphic to $(\Z{2})^2$,}
then up to conjugation, we obtain three groups, namely
\begin{center}$
\begin{array}{lllll}
\mbox{(a)}&\mbox{$G$ is generated by}&(x,y)Ê\mapsto (-x,-y)&\mbox{and}&(x,y) \mapsto (x^{-1},y^{-1}).\\
\mbox{(b)}&\mbox{$G$ is generated by}& (x,y)Ê\mapsto (\pm x,\pm y)&\mbox{and}&(x,y) \mapsto (x^{-1},y^{-1}).\\
\mbox{(c)}&\mbox{$G$ is given by}& \multicolumn{2}{l}{(x,y)Ê\mapsto (\pm x^{\pm 1},\pm y^{\pm 1}).}\end{array}
$\end{center}
The group $[2]$ with $n=1$ is conjugate to (a) by $(x,y)\dasharrow (x,x\frac{y+x}{y+x^{-1}})$. The groups (b) and (c) are respectively equal to $[4]$ and $[5]$.

We now suppose that the group $G$ is not contained in $\PGLn{2} \times \PGLn{2}$. Any element $\varphi\in \Aut(\mathbb{P}^1 \times \mathbb{P}^1)$ not contained in $\PGLn{2} \times \PGLn{2}$ is conjugate to $\varphi:(x,y)\mapsto (\alpha(y),x)$, where $\alpha \in \Aut(\mathbb{P}^1)$, and if $\varphi$ is of finite order, $\alpha$ may be chosen to be $y\mapsto \lambda y$ with $\lambda \in \K^{*}$ a root of unity. 

Thus, up to conjugation, $G$ is generated by the group $H =G\cap (\PGLn{2} \times \PGLn{2})$ and one element $(x,y)\mapsto (\lambda y,x)$, for some $\lambda \in \K^{*}$ of finite order. Since the group $G$ is Abelian, every element of $H$ is of the form $(x,y) \mapsto (\beta(x),\beta(y))$, for some $\beta\in \PGLn{2}$ satisfying $\beta(\lambda x)=\lambda \beta(x)$. Three possibilities occur, depending on the value of $\lambda$ which may be $1$, $-1$ or something else.

\emph{If $\lambda=1$,}
we conjugate the group by some element $(x,y) \mapsto (\gamma(x),\gamma(y))$ so that $H$ is either diagonal or equal to the group generated by $(x,y) \mapsto (-x,-y)$ and $(x,y) \mapsto (x^{-1},y^{-1})$. In the first situation, the group is contained in $(\K^{*})^2 \rtimes \Z{2}$ (which gives $[1]$); the second situation gives $[7]$.

\emph{If $\lambda=-1$,}
the group $H$ contains the square of $(x,y) \mapsto (-y,x)$, which is $(x,y) \mapsto (-x,-y)$ and is either cyclic or generated by $(x,y) \mapsto (-x,-y)$ and $(x,y) \mapsto (x^{-1},y^{-1})$. 
If $H$ is cyclic, it is diagonal, since it contains $(x,y) \mapsto (-x,-y)$, so $G$ is contained in $(\K^{*})^2 \rtimes \Z{2}$. The second possibility gives $[6]$. 

\emph{If $\lambda \not=\pm 1$,}
the group $H$ is diagonal and then $G$ is contained in $(\K^{*})^2 \rtimes \Z{2}$.

We now prove that distinct groups of the list are not birationally conjugate.

First of all, each group of case $[1]$ fixes at least one point of $\mathbb{P}^1\times\mathbb{P}^1$. Since the other groups of the list don't fix any point, they are not conjugate to $[1]$ \cite[Proposition A.2]{bib:KoS}.

Consider the other groups. The set of isomorphic groups are those of cases $[3]$ (with $n=1$), $[4]$ and $[7]$ (isomorphic to $(\Z{2})^3$), and of cases $[2]$ (with $n=2$) and $[6]$ (isomorphic to $\Z{2}\times\Z{4}$).

The groups of cases $[2]$ to $[5]$ leave two pencils of rational curves invariant (the fibres of the two projections $\mathbb{P}^1\times\mathbb{P}^1\rightarrow \mathbb{P}^1$) which intersect freely in exactly one point. We prove that this is not the case for $[6]$ and $[7]$; this shows that these two groups are not birationally conjugate to any of the previous groups. Take $G\subset \Aut(\mathbb{P}^1 \times \mathbb{P}^1)$ to be either $[6]$ or $[7]$. We have then $\Pic{\mathbb{P}^1 \times \mathbb{P}^1}^G=\mathbb{Z}d$, where $d=-\frac{1}{2}K_{\mathbb{P}^1\times\mathbb{P}^1}$ is the diagonal of $\mathbb{P}^1 \times \mathbb{P}^1$.
Suppose that there exist two $G$-invariant pencils $\Lambda_1=n_1 d$ and $\Lambda_2=n_2 d$ of rational curves, for some positive integers $n_1,n_2$ (we identify here a pencil with the class of its elements in $\Pic{\mathbb{P}^1\times\mathbb{P}^1}^G$). The intersection $\Lambda_1 \cdot \Lambda_2=2 n_1 n_2$ is an even integer. Note that the fixed part of the intersection is also even, since $G$ is of order $8$ and acts without fixed points on $\mathbb{P}^1\times\mathbb{P}^1$. The free part of the intersection is then also an even integer and hence is not $1$.

Let us now prove that $[4]$ is not birationally conjugate to $[3]$ (with $n=1$). This follows from the fact that $[4]$ contains three subgroups that are fixed-point free (the groups generated by $(x,y)\mapsto (x^{-1},y^{-1})$ and one of the three involutions of the group $(x,y)\mapsto (\pm x,\pm y)$), whereas $[3]$ (with $n=1$) contains only one such subgroup, which is $(x,y)\mapsto (\pm x^{\pm 1},y)$.

We now prove the last assertion.
The finite Abelian groups of automorphisms of\hspace{0.2 cm}$\Pn$ are conjugate either to $[1]$ or to the group $V_9$, isomorphic to $(\Z{3})^2$ (see Proposition~\ref{Prp:PGL3Cr}). As no group of the list $[2]$ through $[7]$ is isomorphic to $(\Z{3})^2$, we are done.
\end{proof}

{\bf Summary of this section.}\hspace{0.1 cm}{\it We have found that the groups common to the three surfaces $\K^2, \mathbb{P}^2$ and $\mathbb{P}^1\times\mathbb{P}^1$ are the "diagonal" ones (generated by $(x,y)\mapsto (\zeta_n x,y)$ and $(x,y) \mapsto (x,\zeta_m y)$). On $\Pn$ there is only one more group, which is the special group $V_9$, and on $\mathbb{P}^1\times\mathbb{P}^1$ there are $2$ families ($[2]$ and $[3]$) and $4$ special groups ($[4]$, $[5]$, $[6]$ and $[7]$).}

\section{Some facts about automorphisms of conic bundles}
\label{Aut:ConBundles}
We first consider conic bundles without mentioning any group action on them. We recall some classical definitions:
\begin{defi}
Let $S$ be a rational surface and $\pi:S\rightarrow \mathbb{P}^1$ be a 
morphism. We say that the pair $(S,\pi)$ is a \emph{conic bundle} if
a general fibre of $\pi$ is isomorphic to $\mathbb{P}^1$, with a finite number of exceptions: these singular fibres are the union of
smooth rational curves $F_1$ and $F_2$ such that $(F_1)^2=(F_2)^2=-1$ and $F_1\cdot F_2=1$.

Let $(S,\pi)$ and $(\tilde{S},\tilde{\pi})$ be two conic bundles. We say that $\varphi:S \dasharrow \tilde{S}$ is a \defn{birational map of conic bundles} if $\varphi$ is a birational map which sends a general fibre of $\pi$ on a general fibre of $\tilde{\pi}$.

We say that a conic bundle $(S,\pi)$ is \defn{minimal} if any birational morphism of conic bundles $(S,\pi)\rightarrow (\tilde{S},\tilde{\pi})$ is an isomorphism.
\end{defi}
We remind the reader of the following well-known result: 
\begin{lemm}
\label{Lem:SmoothHirz}
Let $(S,\pi)$ be a conic bundle. The following conditions are equivalent:
\begin{itemize}
\item
$(S,\pi)$ is minimal.
\item
The fibration $\pi$ is smooth, i.e.\ no fibre of $\pi$ is singular.
\item
$S$ is a Hirzebruch surface $\mathbb{F}_m$, for some integer $m\geq 0$.\proofend
\end{itemize}
\end{lemm}

Blowing-down one irreducible component in any singular fibre of a conic bundle $(S,\pi)$, we obtain a birational morphism of conic bundles $S\rightarrow \mathbb{F}_m$ for some integer $m\geq 0$. Note that $m$ depends on the choice of the blown-down components. The following lemma gives some information on the possibilities. Note first that since the sections of $\mathbb{F}_m$ have self-intersection $\geq -m$, the self-intersections of the sections of $\pi$ are also bounded from below.
\begin{lemm}
\label{Prp:IskWithoutG}
\index{Conic bundles!automorphisms}
Let $(S,\pi)$ be a conic bundle on a surface $S\not\cong\mathbb{P}^1\times\mathbb{P}^1$. Let $-n$ be the minimal self-intersection of sections of $\pi$ and let $r$ be the number of singular fibres of $\pi$. Then $n\geq 1$ and: 
\begin{enumerate}
\item[\upshape 1.]
There exists a birational morphism of conic bundles $p_{-}:S\rightarrow \mathbb{F}_n$ such that:
\begin{enumerate}
\item[\upshape (a)]
$p_{-}$ is the blow-up of\hspace{0.05 cm} $r$ points of\hspace{0.05 cm} $\mathbb{F}_n$, none of which lies on the exceptional section $E_n$.
\item[\upshape (b)]
The strict pull-back $\widetilde{E_n}$ of $E_n$ by~$p_{-}$ is a section of $\pi$ with self-intersection $-n$.
\end{enumerate}
\item[\upshape 2.]
If there exist two different sections of $\pi$ with self-intersection $-n$, then $r\geq 2n$. In this case, there exist birational morphisms of conic bundles $p_0:S\rightarrow \mathbb{F}_0=\mathbb{P}^1\times\mathbb{P}^1$ and $p_1:S\rightarrow \mathbb{F}_1$.
\end{enumerate}
\end{lemm}
\begin{proof}
We denote by~$s$ a section of $\pi$ of minimal self-intersection $-n$, for some integer~$n$ (this integer is in fact positive, as will appear in the proof). Note that this curve intersects exactly one irreducible component of each singular fibre.

If $r=0$, the lemma is trivially true: take $p_{-}$ to be the identity map. We now suppose that $r\geq1$, and denote by $F_1,...,F_r$ the irreducible components of the singular fibres which do not intersect $s$. Blowing these down, we get a birational morphism of conic bundles $p_{-}:S\rightarrow \mathbb{F}_m$, for some integer $m\geq 0$.
The image of the section $s$ by~$p_{-}$ is a section of the conic bundle of $\mathbb{F}_m$ of minimal self-intersection, so we get $m=n$, and $n \geq 0$. 
If we had $n=0$, then taking some section $\tilde{s}$ of $\mathbb{P}^1\times\mathbb{P}^1$ of self-intersection $0$ passing through at least one blown-up point, its strict pull-back by~$p_{-}$ would be a section of negative self-intersection, which contradicts the minimality of $s^2=-n=0$.
We find finally that $m=n>0$, and that $p_{-}(s)$ is the unique section $\mathbb{F}_n$ of self-intersection $-n$. This proves the first assertion. 

We now prove the second assertion. Suppose that some section $t\not=s$ has self-intersection $-n$. The Picard group of $S$ is generated by $s=p_{-}^{*}(E_n)$, the divisor $f$ of a fibre of $\pi$ and $F_1,...,F_r$. Write $t$ as $t=s+bf-\sum_{i=1}^r{a_i}F_i$, for some integers $b,a_1,...,a_r$, with $a_1,...,a_r \geq 0$. We have $t^2=-n$ and $t\cdot (t+K_S)=-2$ (adjunction formula), where $K_S=p_{-}^{*}(K_{\mathbb{F}_n})+\sum_{i=1}^r F_i=-(n+2)f-2s+\sum_{i=1}^r F_i$. These relations give:
\begin{center}
$\begin{array}{ccccl}
s^2&=&t^2&=&s^2-\sum_{i=1}^r a_i^2+2b,\vspace{0.1 cm}\\
n-2&=&t\cdot K_S&=&-(n+2)+2n-2b+\sum_{i=1}^r a_i,
\end{array}$\end{center}
whence $\sum_{i=1}^r a_i=\sum_{i=1}^r a_i^2=2b$, so each $a_i$ is equal to $0$ or $1$ and consequently $2b\leq r$. Since $s\cdot t=b-n \geq 0$, we find that $r\geq 2n$, as announced.

Finally, by contracting $f-F_1,f-F_2,...,f-F_n,F_{n+1},F_{n+2},...,F_{r}$, we obtain a birational morphism $p_0$ of conic bundles which sends $s$ on a section of self-intersection $0$ and whose image is thus $\mathbb{F}_0$. Similarly, the morphism $p_1:S\rightarrow \mathbb{F}_1$ is given by the contraction of $f-F_1,f-F_2,...,f-F_{n-1},F_{n},F_{n+1},...,F_{r}$.
\end{proof}

We now add some group actions on the conic bundles, and give natural definitions (note that we will restrict ourselves to finite or Abelian groups only when this is needed and will then say so):
\begin{defi}
Let $(S,\pi)$ be some conic bundle. 
\begin{itemize}
\item
We denote by $\Aut(S,\pi)\subset \Aut(S)$ the group of automorphisms of the conic bundle, i.e.\ automorphisms of $S$ that send a general fibre of $\pi$ on another general fibre.
\end{itemize}
Let $G \subset \Aut(S,\pi)$ be some group of automorphisms of the conic bundle $(S,\pi)$.
\begin{itemize}
\item
We say that a birational map of conic bundles $\varphi:S \dasharrow \tilde{S}$ is \defn{$G$-equivariant} if the $G$-action on $\tilde{S}$ induced by~$\varphi$ is biregular (it is clear that it preserves the conic bundle structure).
\item
We say that the triple $(G,S,\pi)$ is \defn{minimal} if any $G$-equivariant birational morphism of conic bundles $\varphi:S\rightarrow \tilde{S}$ is an isomorphism.\end{itemize}\end{defi}
\begin{rema}
We insist on the fact that since a \emph{conic bundle} is for us a pair $(S,\pi)$, an automorphism of $S$ is \emph{not} necessarily an automorphism of the conic bundle (i.e.\ $\Aut(S)\not=\Aut(S,\pi)$ in general). 
One should be aware that in the literature, \emph{conic bundle} sometimes means "a variety admitting a conic bundle structure". 
\end{rema}
\begin{rema}
If $G \subset \Aut(S,\pi)$ is such that the pair $(G,S)$ is minimal, so is the triple $(G,S,\pi)$. The converse is not true in general (see Remark~\ref{remark:Exakappanotmin}).
\end{rema}
Note that any automorphism of the conic bundle acts on the set of singular fibres and on its irreducible components. 
The permutation of the two components of a singular fibre is very important (Lemma~\ref{Lem:MinTripl}). For this reason, we introduce some terminology:
\begin{defi}
Let $g\in \Aut(S,\pi)$ be an automorphism of the conic bundle $(S,\pi)$. Let $F=\{F_1,F_2\}$ be a singular fibre. We say that~$g$ \defn{twists} the singular fibre $F$ if $g(F_1)=F_2$ {\upshape (}and consequently $g(F_2)=F_1${\upshape )}. 

If~$g$ twists at least one singular fibre of $\pi$, we will say that~$g$ \emph{twists} the conic bundle $(S,\pi)$, or simply (if the conic bundle is implicit) that~$g$ is a \emph{twisting} element.
\end{defi}

Here is a simple but very important observation:
\begin{lemm}
\label{Lem:MinTripl}
 Let $G \subset \Aut(S,\pi)$ be a group of automorphisms of a conic bundle.
The following conditions are equivalent:
\begin{itemize}
\item[\upshape 1.]
The triple $(G,S,\pi)$ is minimal.
\item[\upshape 2.]
Any singular fibre of $\pi$ is twisted by some element of $G$.\hfill \proofend
\end{itemize}
\end{lemm}
\begin{rema}
An automorphism of a conic bundle with a non-trivial action on the basis of the fibration may twist at most two singular fibres. However, an automorphism with a trivial action on the basis of the fibration may twist a large number of fibres. We will give in Propositions~\ref{Prp:DescriptionTwistingElementsFinite} and~\ref{Prp:DescriptionTwistingElementsInfinite} a precise description of all twisting elements.
\end{rema}

\bigskip

The following lemma is a direct consequence of Lemma~\ref{Prp:IskWithoutG}; it provides information on the structure of the underlying variety of a conic bundle admitting a twisting automorphism.
\begin{lemm}
\label{Lem:GoingToF0F1}
Suppose that some automorphism of the conic bundle $(S,\pi)$ twists at least one singular fibre. Then, the following occur.
\begin{itemize}
\item[\upshape 1.]
There exist two birational morphisms of conic bundles $p_0: S \rightarrow \mathbb{F}_0$ and $p_1:S \rightarrow \mathbb{F}_1$ (which are not $g$-equivariant).
\item[\upshape 2.]
Let $-n$ be the minimal self-intersection of sections of $\pi$ and let $r$ be the number of singular fibres of $\pi$. Then, $r\geq 2n\geq 2$.
\end{itemize}
\end{lemm}
\begin{proof}
Note that any section of $\pi$ touches exactly one component of each singular fibre. Since~$g$ twists some singular fibre, its action on the set of sections of $S$ is fixed-point-free. The number of sections of minimal self-intersection is then greater than $1$ and we apply Lemma~\ref{Prp:IskWithoutG} to get the result.
\end{proof}
\begin{rema}
A result of the same kind can be found in \cite{bib:Isk1}, Theorem~1.1.\end{rema}

\bigskip

\begin{lemm}
\label{Lem:Degree567CB}
Let $G \subset \Aut(S,\pi)$ be a group of automorphisms of the conic bundle $(S,\pi)$, such that:
\begin{itemize}
\item
 $\pi$ has at most $3$ singular fibres (or equivalently $(K_S)^2\geq 5$);
 \item
the triple $(G,S,\pi)$ is minimal.
 \end{itemize}
 Then, $S$ is either a Hirzeburch surface or a del Pezzo surface of degree $5$ or $6$, depending on whether the number of singular fibres is $0$, $3$ or $2$ respectively.
\end{lemm}
\begin{proof}
Let $-n$ be the minimal self-intersection of sections of $\pi$ and let $r\leq 3$ be the number of singular fibres of $\pi$. If $r=0$, we are done, so we may suppose that $r>0$.
Since $(G,S,\pi)$ is minimal, every singular fibre is twisted by some element of $G$ (Lemma~\ref{Lem:MinTripl}). From Lemma~\ref{Lem:GoingToF0F1}, we get $r\geq 2n\geq 2$, whence $r =2$ or $3$ and $n=1$, and we obtain the existence of some birational morphism of conic bundles (not $G$-equivariant) $p_1:S\rightarrow \mathbb{F}_1$. So the surface $S$ is obtained by the blow-up of $2$ or $3$ points of $\mathbb{F}_1$, not on the exceptional section (Lemma~\ref{Prp:IskWithoutG}), and thus by blowing-up $3$ or $4$ points of\hspace{0.2 cm}$\Pn$, no $3$ of which are collinear (otherwise we would have a section of self-intersection $\leq -2$). The surface is then a del Pezzo surface of degree~$6$ or~$5$.
\end{proof}

\bigskip

\begin{rema}
\label{Rem:ExactSeq}
We conclude this section by mentioning an important exact sequence. Let $G \subset \Aut(S,\pi)$ be some group of automorphisms of a conic bundle $(S,\pi)$. We have a natural homomorphism
$\overline{\pi}: G
\rightarrow \Aut(\mathbb{P}^1)= \PGLn{2}$ that satisfies $\overline{\pi}(g)\pi=\pi g$, for every $g \in G$.
We observe that the group $G' =\ker \overline{\pi}$ of automorphisms 
that leave every fibre invariant embeds in the group
$\PGL(2,\K(x))$ of automorphisms of the generic fibre $\mathbb{P}^1(\K(x))$.
Then we get the exact sequence
\begin{equation}
\label{eq:ExactSeqCB}
1 \rightarrow G' \rightarrow G \stackrel{\overline{\pi}}{\rightarrow} \overline{\pi}(G) \rightarrow 1.\end{equation} This restricts the structure of $G$;
for example if $G$ is Abelian and finite, so are $G'$ and $\overline{\pi}(G)$, and we know that the finite Abelian subgroups of $\PGLn{2}$ and $\PGL(2,\K(x))$ are 
either cyclic or isomorphic to $(\Z{2})^2$.

We also see that the group $G$ is birationally conjugate to a subgroup of the group of birational transformations of $\mathbb{P}^1\times\mathbb{P}^1$ of the form (written in affine coordinates):
$$(x,y)\dasharrow \left(\frac{ax+b}{cx+d},\frac{\alpha(x)y+\beta(x)}{\gamma(x)y+\delta(x)}\right),$$
where $a,b,c,d \in \K$, $\alpha,\beta,\gamma,\delta \in \K(x)$, and $(ad-bc)(\alpha\delta-\beta\gamma)\not=0$.

This group, called the \defn{de Jonqui\`eres group}, is the group of birational transformations of $\mathbb{P}^1\times\mathbb{P}^1$ that preserve the fibration induced by the first projection, and is isomorphic to $\PGL(2,\K(x))\rtimes \PGL(2,\K)$.

The subgroups of this group can be studied algebraically (as in \cite{bib:Bea2} and \cite{bib:JBSMF}) but we will not adopt this point of view here.
\end{rema}
\section{The del Pezzo surface of degree $6$}
\label{Sec:DelPezzo6}
There is a single isomorphism class of del Pezzo surfaces of degree $6$, since all sets of three non-collinear points of\hspace{0.2 cm}$\Pn$ are equivalent under the action of linear automorphisms. 
Consider the surface $S_6$ of degree $6$ defined by the blow-up of the points $A_1=(1:0:0)$, $A_2=(0:1:0)$ and $A_3=(0:0:1)$. 
We may view it in $\Pn \times \Pn$, defined as $\{ \big((x:y:z) , (u:v:w)\big)\ | \ ux=vy=wz\}$, where the blow-down $p:S_6\rightarrow \Pn$ is the restriction of the projection on one copy of\hspace{0.2 cm}$\Pn$, explicitly $p: \big((x:y:z) , (u:v:w)\big)\mapsto (x:y:z)$.
There are exactly $6$ exceptional divisors, which are the pull-backs of the $A_i$'s by the two projection morphisms. We write $E_i=p^{-1}(A_i)$ and denote by $D_{ij}$ the strict pull-back by $p$ of the line of $\Pn$ passing through $A_i$ and $A_j$.

The group of automorphisms of $S_6$ is well known (see for example \cite{bib:Wim}, \cite{bib:Dol}). It is isomorphic to $(\K^{*})^2 \rtimes (\Sym_3 \times \Z{2})$, where $(\K^{*})^2 \rtimes \Sym_3$ is the lift on $S_6$ of the group of automorphisms of\hspace{0.2 cm}$\Pn$ that leave the set $\{A_1,A_2,A_3\}$ invariant, and $\Z{2}$ is generated by the permutation of the two factors (it is the lift of the standard quadratic transformation $(x:y:z) \dasharrow (yz:xz:xy)$ of $\Pn$); the action of $\Z{2}$ on $(\K^{*})^2$ sends an element on its inverse.

There are three conic bundle structures on the surface $S_6$. 
Let $\pi_1:S_6\rightarrow \mathbb{P}^1$ be the morphism defined by
\begin{center}
$\pi_1:\big((x:y:z) , (u:v:w)\big)\mapsto\left\{\begin{array}{lll}
(y:z)& \mbox{ if } &(x:y:z) \not= (1:0:0),\\
(w:v)& \mbox{ if } &(u:v:w) \not= (1:0:0).\end{array}\right.$
\end{center}
Note that $p$ sends the fibres of $\pi_1$ on lines of\hspace{0.2 cm}$\Pn$ passing through $A_1$. There are exactly two singular fibres of this fibration, namely
\begin{center}
\begin{tabular}{lll}
$\pi_1^{-1}(1:0)=\{E_2,D_{12}\}$& and &$\pi_1^{-1}(0:1)=\{E_3,D_{13}\}$;
\end{tabular}\end{center}
and $E_1$, $D_{23}$ are sections of $\pi_1$.

\begin{center}
\includegraphics[width=40.00mm]{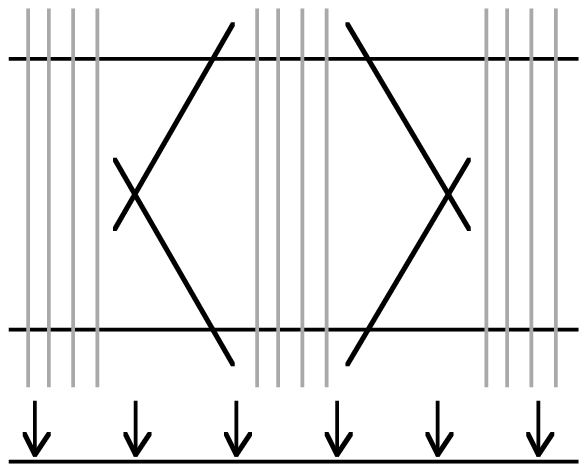}%
\drawat{0mm}{28mm}{$D_{23}$}%
\drawat{-12mm}{24.5mm}{$E_3$}%
\drawat{-12mm}{13.1mm}{$D_{13}$}%
\drawat{0mm}{9mm}{$E_1$}%
\drawat{-34.1mm}{13.1mm}{$D_{12}$}%
\drawat{-32.8mm}{24.5mm}{$E_2$}%
\drawat{-22mm}{2.5mm}{$\pi_1$}%
\end{center}
\begin{lemm}
\label{Lem:AutDP6CB}
The group $\Aut(S_6,\pi_1)$ of automorphisms of the conic bundle $(S_6,\pi_1)$ acts on the hexagon $\{E_1,E_2,E_3,D_{12},D_{13},D_{23}\}$ and leaves the set $\{E_1,D_{23}\}$ invariant.
\begin{enumerate}
\item[\upshape 1.]
The action on the hexagon gives rise to the exact sequence 
\begin{center}
$1\rightarrow (\K^{*})^2 \rightarrow \Aut(S_6,\pi_1) \rightarrow (\Z{2})^2 \rightarrow 1.$
\end{center}
\item[\upshape 2.]
This exact sequence is split and $\Aut(S_6,\pi_1)= (\K^{*})^2\rtimes (\Z{2})^2$, where 
\begin{enumerate}
\item[\upshape (a)]
$(\K^{*})^2$ is the group of automorphisms of the form\\
$\big((x:y:z) , (u:v:w)\big) \mapsto \big((x:\alpha y: \beta z) ,(\alpha\beta u:\beta v:\alpha w)\big)$,
$\alpha,\beta \in \K^{*}$.
\item[\upshape (b)]
The group $ (\Z{2})^2$ is generated by the automorphisms
\begin{center}
$\big((x:y:z) , (u:v:w)\big) \mapsto \big((x:z:y),(u:w:v)\big)$,
\end{center}
whose action on the set of exceptional divisors is $(E_2 \ E_3)(D_{12}\ D_{13})$; and 
\begin{center}
$\big((x:y:z) , (u:v:w)\big) \mapsto \big((u:v:w),(x:y:z)\big),$
\end{center}
whose action is $(E_1 \ D_{23})(E_2\ D_{13})(E_3\ D_{12})$.
\item[\upshape (c)]
The action of $(\Z{2})^2$ on $(\K^{*})^2$ is generated by permutation of the coordinates and inversion.\end{enumerate}
\end{enumerate}
\end{lemm}
\begin{proof}\upshape
Since $\Aut(S_6)$ acts on the hexagon, so does $\Aut(S_6,\pi_1) \subset \Aut(S_6)$. Since the group $\Aut(S_6,\pi_1)$ sends a section on a section, the set $\{E_1,D_{23}\}$ is invariant.

The group $(\K^{*})^2$ leaves the conic bundle invariant, and is the kernel of the action of $\Aut(S_6,\pi_1)$ on the hexagon. As the set $\{E_1,D_{23}\}$ is invariant, the image is contained in the group $(\Z{2})^2$ generated by $(E_2 \ E_3)(D_{12}\ D_{13})$ and $(E_1 \ D_{23})(E_2\ D_{13})(E_3\ D_{12})$. The rest of the lemma follows directly.
\end{proof}

By permuting coordinates, we have two other conic bundle structures on the surface $S_6$, given by the following morphisms $\pi_2,\pi_3:S_6 \rightarrow \mathbb{P}^1$:
\begin{center}
$\pi_2(\big((x:y:z) , (u:v:w)\big))=\left\{\begin{array}{lll}
(x:z)& \mbox{ if }& (x:y:z) \not= (0:1:0),\\
(w:u)& \mbox{ if }& (u:v:w) \not= (0:1:0).\end{array}\right.$\vspace{0.1cm}\\
$\pi_3(\big((x:y:z) , (u:v:w)\big))=\left\{\begin{array}{lll}
(x:y)& \mbox{ if }& (x:y:z) \not= (0:0:1),\\
(v:u)& \mbox{ if }& (u:v:w) \not= (0:0:1).\end{array}\right.$
\end{center}

The description of the exceptional divisors on $S_6$ shows that $\pi_1,\pi_2$ and $\pi_3$ are the only conic bundle structures on $S_6$.

\begin{lemm}
\label{Lem:DP6CBPasMin}
For $i=1,2,3$, the pair $(\Aut(S_6,\pi_i),S_6)$ is not minimal. More precisely the morphism $\pi_j\times \pi_k:S_6\rightarrow \mathbb{P}^1\times\mathbb{P}^1$ conjugates $\Aut(S_6,\pi_i)$ to a subgroup of $\Aut(\mathbb{P}^1\times \mathbb{P}^1)$, where $\{i,j,k\}=\{1,2,3\}$.
\end{lemm}
\begin{proof}\upshape
The union of the sections $E_1$ and $D_{23}$ is invariant by the action of the whole group $\Aut(S_6,\pi_1)$. Since these two exceptional divisors don't intersect, we can contract both and get a birational $\Aut(S_6,\pi_1)$-equivariant morphism from $S_6$ to $\mathbb{P}^1 \times \mathbb{P}^1$: the pair $(\Aut(S_6,\pi_1),S_6)$ is thus not minimal; 
explicitly, the birational morphism is given by $q \mapsto (\pi_2(q),\pi_3(q))$, as stated in the lemma.
We obtain the other cases by permuting coordinates.
\end{proof}
\begin{rema}\label{Rem:DP6CBPasMin} The subgroup of $\Aut(\mathbb{P}^1 \times \mathbb{P}^1)$ obtained in this manner doesn't leave any of the two fibrations of $\mathbb{P}^1 \times \mathbb{P}^1$ invariant.\end{rema}
\begin{coro}
\label{Coro:S6cnotminimal}
If $(G,S_6)$ is a minimal pair (where $G\subset \Aut(S_6)$), then $G$ does not preserve any conic bundle structure.\proofend
\end{coro}
\bigskip

We conclude this section with a fundamental example; we will use several times the following automorphism $\kappa_{\alpha,\beta}$ of $(S_6,\pi_1)$:
\begin{exem}
\label{Exa:KappaAB}
For any $\alpha,\beta \in \K^{*}$, we define $\kappa_{\alpha,\beta}$ to be the following automorphism of $(S_6,\pi_1)$:
\begin{center}
$\kappa_{\alpha,\beta}:\big((x:y:z) , (u:v:w)\big) \mapsto \big((u:\alpha w:\beta v), (x:\alpha^{-1} z:\beta^{-1} y)\big)$.\end{center}

Note that $\kappa_{\alpha,\beta}$ twists the two singular fibres of $\pi_1$ (see Lemma~\ref{Lem:KappaDP6} below); its action on the basis of the fibration is $(x_1:x_2) \mapsto (\alpha x_1:\beta x_2)$ and 
\begin{center}$\kappa_{\alpha,\beta}^2(\big((x:y:z) , (u:v:w)\big))=\big((x:\alpha\beta^{-1} y:\alpha^{-1}\beta z), (u:\alpha^{-1}\beta v:\alpha\beta^{-1} w)\big)$.\end{center} So $\kappa_{\alpha,\beta}$ is an involution if and only if its action on the basis of the fibration is trivial.
\end{exem}
\begin{lemm}
\label{Lem:KappaDP6}
Let $g \in \Aut(S_6,\pi_1)$ be an automorphism of the conic bundle $(S_6,\pi_1)$. The following conditions are equivalent:
\begin{itemize}
\item
the triple $(<g>,S_6,\pi_1)$ is minimal;
\item
$g$ twists the two singular fibres of $\pi_1$;
\item
the action of~$g$ on the exceptional divisors of $S_6$ is $(E_1 \ D_{23})(E_2\ D_{12})(E_3\ D_{13})$;
\item
$g=\kappa_{\alpha,\beta}$ for some $\alpha,\beta \inÊ\K^{*}$.
\end{itemize}
\end{lemm}
\begin{proof}
According to Lemma~\ref{Lem:AutDP6CB} the action of $\Aut(S_6,\pi_1)$ on the exceptional curves is isomorphic to $(\Z{2})^2$ and hence the possible actions of $g\not=1$ are these:

\begin{tabular}{ll}
1. $\mathrm{id}$,&\hspace{1 cm}2. $(E_2 \ E_3)(D_{12}\ D_{13}),$\\
3. $(E_1 \ D_{23})(E_2\ D_{13})(E_3\ D_{12})$, &\hspace{1 cm}4. $(E_1 \ D_{23})(E_2\ D_{12})(E_3\ D_{13}).$
\end{tabular}

In the first three cases, the triple $(<g>,S_6,\pi_1)$ is not minimal. Indeed, the blow-down of $\{E_2,E_3\}$ or $\{E_2,D_{13}\}$ gives a $g$-equivariant birational morphism of conic bundles.

Hence, if $(<g>,S_6,\pi_1)$ is minimal, its action on the exceptional curves is the fourth one above, as stated in the lemma, and it then twists the two singular fibres of $\pi_1$. Conversely if~$g$ twists the two singular fibres of $\pi_1$, the triple $(<g>,S_6,\pi_1)$ is minimal (by Lemma~\ref{Lem:MinTripl}).

It remains to see that the last assertion is equivalent to the others. This follows from Lemma~\ref{Lem:AutDP6CB}; indeed this lemma implies that $(\K^{*})^2 \kappa_{1,1}$ is the set of elements of $\Aut(S_6,\pi_1)$ inducing the permutation $(E_1 \ D_{23})(E_2\ D_{12})(E_3\ D_{13})$.
\end{proof}

\begin{rema}
\label{remark:Exakappanotmin}
The pair $(\Aut(S_6,\pi_1),S_6)$ is not minimal (Lemma 
\ref{Lem:DP6CBPasMin}). Consequently $<\kappa_{\alpha,\beta}>$ is an example of a group whose action on the surface is not minimal, but whose action on a conic bundle is minimal.
\end{rema}
\section{The del Pezzo surface of degree $5$}
\label{Sec:DelPezzo5}
As for the del Pezzo surface of degree $6$, there is a single isomorphism class of del Pezzo surfaces of degree $5$. 
Consider the del Pezzo surface $S_5$ of degree $5$ defined by the blow-up $p:S_5 \rightarrow \Pn$ of the points $A_1=(1:0:0)$, $A_2=(0:1:0)$, $A_3=(0:0:1)$ and $A_4=(1:1:1)$. 
There are $10$ exceptional divisors on $S_5$, namely the divisor $E_i=p^{-1}(A_i)$, for $i=1,...,4$, and the strict pull-back $D_{ij}$ of the line of $\Pn$ passing through $A_i$ and $A_j$, for $1\leq i<j\leq 4$. There are $5$ sets of $4$ skew exceptional divisors on $S_5$, namely

\begin{center}
$\begin{array}{lll}
F_1=\{E_1,D_{23},D_{24},D_{34}\},& F_2=\{E_2,D_{13},D_{14},D_{34}\}, & 
F_3=\{E_3,D_{12},D_{14},D_{24}\},\\
F_4=\{E_4,D_{12},D_{13},D_{23}\}, & F_5=\{E_1,E_2,E_3,E_4\}.\end{array}$
\end{center}

\begin{prop}
\label{Prp:AutS5}
The action of $\Aut(S_5)$ on the five sets $F_1,...,F_5$ of four skew exceptional divisors of $S_5$ gives rise to an isomomorphism 
\begin{center}$\rho:\Aut(S_5) \rightarrow \Sym_5$.\end{center}
Furthermore, the actions of $\Sym_n$, $\Alt_m\subset \Aut(S_5)$ on $S_5$ given by the canonical embedding of these groups into $\Sym_5$ are fixed-point free if and only if $n=3,4,5$, respectively $m=4,5$.
\end{prop}
\begin{proof}
Since any automorphism in the kernel of $\rho$ leaves $E_1,E_2,E_3$ and $E_4$ invariant and hence is the lift of an automorphism of $\Pn$ that fixes the $4$ points, the homomorphism $\rho$ is injective.
 
 We now prove that $\rho$ is also surjective. Firstly, the lift of the group of automorphisms of $\Pn$ that leave the set $\{A_1,A_2,A_3,A_4\}$ invariant is sent by $\rho$ on $\Sym_4=\Sym_{\{F_1,F_2,F_3,F_4\}}$. Secondly, the lift of the standard quadratic transformation $(x:y:z) \dasharrow (yz:xz:xy)$ is an automorphism of $S_5$, as its lift on $S_6$ is an automorphism, and as it fixes the point $A_4$; its image by $\rho$ is $(F_4\ F_5)$.
 
 It remains to prove the last assertion. First of all, it is clear that the actions of the cyclic groups $\Alt_3$ and $\Sym_2$ fix some points. The group $\Sym_3\subset \Aut(\Pn)$ of permutations of $A_1,A_2$ and $A_3$ fixes exactly one point, namely $(1:1:1)$. The blow-up of this point gives a fixed-point free action on $\mathbb{F}_1$, and thus its lift on $S_5$ is also fixed-point free. The group $\Alt_4\subset \Aut(\Pn)$ contains the element $(x:y:z)\mapsto (z:x:y)$ (which corresponds to $(1\ 2\ 3)$) that fixes exactly three points, i.e.\ $(1:a:a^2)$ for $a^3=1$. It also contains the element $(x:y:z)\mapsto (z-y:z-x:z)$ (which corresponds to $(1\ 2)(3\ 4)$) that does not fix $(1:a:a^2)$ for $a^3=1$. Thus, the action of $\Alt_4$ on $\Pn$ is fixed-point free and the same is true on $S_5$.
 \end{proof}
\begin{rema}
The structure of $\Aut(S_5)$ is classical and can be found for example in \cite{bib:Wim} and \cite{bib:Dol}.
\end{rema}

\begin{lemm}\label{Lem:AutDP5Conic}
Let $\pi:S_5\rightarrow \mathbb{P}^1$ be some morphism inducing a conic bundle $(S_5,\pi)$. There are exactly four exceptional curves of $S_5$ which are sections of $\pi$; the blow-down of these curves gives rise to a birational morphism $p:S_5 \rightarrow \Pn$ which conjugates the group $\Aut(S_5,\pi)\cong \Sym_4$ to the subgroup of $\Aut(\Pn)$ that leaves invariant the four points blown-up by $p$. In particular, the pair $(\Aut(S_5,\pi),S_5)$ is not minimal.
\end{lemm}
\begin{proof}
Blowing-down one component in any singular fibre, we obtain a birational morphism of conic bundles (not $\Aut(S_5,\pi)$-equivariant) from $S_5$ to some Hirzebruch surface $\mathbb{F}_n$. Since $S_5$ does not contain any curves of self-intersection $\leq -2$,~$n$ is equal to $0$ or $1$. Changing the component blown-down in a singular fibre performs an elementary link $\mathbb{F}_n\dasharrow \mathbb{F}_{n\pm 1}$; we may then assume that $n=1$, and that $\mathbb{F}_1$ is the blow-up of $A_1\in \Pn$. Consequently, the fibres of the conic bundles correspond to the lines passing through $A_1$. Denoting by $A_2,A_3,A_4$ the other points blown-up by the constructed birational morphism $S_5\rightarrow \Pn$ and using the same notation as before, the three singular fibres are $\{E_i,D_{1i}\}$ for $i=2,...,4$, and the other exceptional curves are four skew sections of the conic bundle, namely the elements of $F_1=\{E_1,D_{23},D_{24},D_{34}\}$. The blow-down of $F_1$ gives an $\Aut(S_5,\pi)$-equivariant birational morphism (that is not a morphism of conic bundles) $p:S_5\rightarrow \Pn$ and conjugates $\Aut(S_5,\pi)$ to a subgroup of the group $\Sym_4\subset\Aut(\Pn)$ of automorphisms that leaves the four points blown-up by $p$ invariant. The fibres of $\pi$ are sent on the conics passing through the four points, so the lift of the whole group $\Sym_4$ belongs to $\Aut(S_5,\pi)$.
\end{proof}
\begin{coro}
\label{Cor:No3singfibres}
Let $G$ be some group of automorphisms of a conic bundle $(S,\pi)$ such that the pair $(G,S)$ is minimal and $(K_S)^2\geq 5$ (or equivalently such that the number of singular fibres of $\pi$ is at most $3$). 
Then, the fibration is smooth, i.e.\ $S$ is a Hirzebruch surface.
\end{coro}
\begin{proof}
Since $(G,S)$ is minimal, so is the triple $(G,S,\pi)$. By Lemma~\ref{Lem:Degree567CB}, the surface $S$ is either a Hirzebruch surface, or a del Pezzo surface of degree $5$ or $6$. Corollary~\ref{Coro:S6cnotminimal} shows that the del Pezzo surface of degree $6$ is not possible and Lemma~\ref{Lem:AutDP5Conic} eliminates the possibility of the del Pezzo surface of degree $5$.
\end{proof}
\section{Description of twisting elements}
In this section, we describe the twisting automorphisms of conic bundles, which are the most important automorphisms (see Lemma~\ref{Lem:MinTripl}).
\begin{lemm}[Involutions twisting a conic bundle]
\label{Lem:DeJI}
Let $g \in \Aut(S,\pi)$ be a twisting automorphism of the conic bundle $(S,\pi)$.
Then, the following properties are equivalent:
\begin{itemize}
\item[\upshape 1.]
$g$ is an involution;
\item[\upshape 2.]
$\overline{\pi}(g)=1$, i.e.\ $g$ has a trivial action on the basis of the fibration;
\item[\upshape 3.]
the set of points of $S$ fixed by~$g$ is an irreducible hyperelliptic curve of genus $(k-1)$ -- a double covering of $\mathbb{P}^1$ by means of $\pi$, ramified over $2k$ points -- plus perhaps a finite number of isolated points, which are the singular points of the singular fibres not twisted by~$g$.
\end{itemize}
Furthermore, if the three conditions above are satisfied, the number of singular fibres of $\pi$ twisted by~$g$ is $2k\geq 2$.
\end{lemm}
\begin{proof}
$1\Rightarrow 2$: 
By contracting some exceptional curves, we may assume that the triple $(<g>,S,\pi)$ is minimal. Suppose that~$g$ is an involution and $\overline{\pi}(g)\not=1$. Then~$g$ may twist only two singular fibres, which are the fibres of the two points of $\mathbb{P}^1$ fixed by $\overline{\pi}(g)$. Hence, the number of singular fibres is $\leq 2$. Lemma~\ref{Lem:Degree567CB} tells us that $S$ is a del Pezzo surface of degree $6$ and then Lemma~\ref{Lem:KappaDP6} shows that $g=\kappa_{\alpha,\beta}$ (Example~\ref{Exa:KappaAB}) for some $\alpha,\beta\in \K^{*}$. But such an element is an involution if and only if it acts trivially on the basis of the fibration.

$(1 \mbox{ and } 2)\Rightarrow 3$: 
Suppose first that $(<g>,S,\pi)$ is minimal. This implies that~$g$ twists every singular fibre of $\pi$.
Therefore, since $\overline{\pi}(g)=1$ and $g^2=1$, on a singular fibre there is one point fixed by~$g$ (the singular point of the fibre) and on a general fibre there are two fixed points. The set of points of $S$ fixed by~$g$ is thus a smooth irreducible curve. The projection $\pi$ gives it as a double covering of $\mathbb{P}^1$ ramified over the points whose fibres are singular and twisted by~$g$. By the Riemann-Hurwitz formula, this number is even, equal to $2k$ and the genus of the curve is $k-1$.

The situation when $(<g>,S,\pi)$ is not minimal is obtained from this one, by blowing-up some fixed points. This adds in each new singular fibre (not twisted by the involution) an isolated point, which is the singular point of the singular fibre.
We then get the third assertion and the final remark.

$3\Rightarrow 2$: This implication is clear.

$2\Rightarrow 1$: If $\overline{\pi}(g)=1$, then, $g^2$ leaves every component of every singular fibre of $\pi$ invariant.
Let $p_1:S\rightarrow \mathbb{F}_1$ be the birational morphism of conic bundles given by Lemma~\ref{Lem:GoingToF0F1}; it is a $g^2$-equivariant birational morphism which conjugates $g^2$ to an automorphism of $\mathbb{F}_1$ that necessarily fixes the exceptional section. The pull-back by $p_1$ of this section is a section $C$ of $\pi$, fixed by $g^2$. Since $C$ touches exactly one component of each singular fibre (in particular those that are twisted by~$g$), $g$~sends $C$ on another section $D$ also fixed by $g^2$. The union of the sections $D$ and $C$ intersects a general fibre in two points, which are exchanged by the action of~$g$. This implies that~$g$ has order $2$.
\end{proof}

\bigskip

We now give some further simple results on twisting involutions.
\begin{coro}
\label{Cor:NoRoot}
Let $(S,\pi)$ be some conic bundle. No involution twisting $(S,\pi)$ has a root in $\Aut(S,\pi)$ which acts trivially on the basis of the fibration.
\end{coro}
\begin{proof}\upshape
Such a root must twist a singular fibre and so (Lemma~\ref{Lem:DeJI}) is an involution.
\end{proof}
\begin{rema}There may exist some roots in $\Aut(S,\pi)$ of twisting involutions which act non trivially on the basis of the fibration.\\ Take for example four general points $A_1,...,A_4$ of the plane and denote by $g\in \Aut(\Pn)$ the element of order $4$ that permutes these points cyclically. The blow-up of these points conjugates~$g$ to an automorphism of the del Pezzo surface $S_5$ of degree $5$ (see Section~\ref{Sec:DelPezzo5}). The pencil of conics of $\Pn$ passing through the four points induces a conic bundle structure on $S_5$, with three singular fibres which are the lift of the pairs of two lines passing through the points. The lift on $S_5$ of~$g$ is an automorphism of the conic bundle whose square is a twisting involution.\end{rema}
\begin{coro}
\label{Cor:CurveFixDeJ}
Let $(S,\pi)$ be some conic bundle and let $g \in \Aut(S,\pi)$. The following conditions are equivalent.
\begin{enumerate}
\item[\upshape 1.]
$g$ twists more than $2$ singular fibres of $\pi$.
\item[\upshape 2.]
~$g$ fixes a curve of positive genus. 
\end{enumerate}
And these conditions imply that~$g$ is an involution which acts trivially on the basis of the fibration and twists at least $4$ singular fibres.
\end{coro}
\begin{proof}\upshape
The first condition implies that~$g$ acts trivially on the basis of the fibration, and thus (by Lemma~\ref{Lem:DeJI}) that~$g$ is an involution which fixes a curve of positive genus.

Suppose that~$g$ fixes a curve of positive genus. Then,~$g$ acts trivially on the basis of the fibration, and fixes $2$ points on a general fibre.
Consequently, the curve fixed by~$g$ is a smooth hyperelliptic curve; we get the remaining assertions from Lemma~\ref{Lem:DeJI}.
\end{proof}

\bigskip

As we mentioned above, the automorphisms that twist some singular fibre are fundamental (Lemma~\ref{Lem:MinTripl}). We now describe  these elements and prove that the only possibilities are twisting involutions, roots of twisting involutions (of even or odd order) and elements of the form $\kappa_{\alpha,\beta}$ (see Example~\ref{Exa:KappaAB}):

\begin{prop}[Classification of twisting elements of finite order]
\label{Prp:DescriptionTwistingElementsFinite}
Let $g \in \Aut(S,\pi)$ be a twisting automorphism of finite order of a conic bundle $(S,\pi)$. Let~$n$ be the order of its action on the basis.

Then $g^n$ is an involution that acts trivially on the basis of the fibration and twists an even number $2k$ of singular fibres; furthermore, exactly one of the following situations occurs:
\begin{itemize}
\item[\upshape 1.]
{\bf ${n=1}$.}
\item[\upshape 2.]
{ ${n>1}$ and ${k=0}$;} in this case~$n$ is even and 
there exists a $g$-equivariant birational morphism of conic bundles $\eta:S\rightarrow S_6$ (where $S_6$ is the del Pezzo surface of degree $6$) such that $\eta g\eta^{-1}=\kappa_{\alpha,\beta}$ for some $\alpha,\beta \in \K^{*}$ (see Example~$\ref{Exa:KappaAB}$). 
\item[\upshape 3.]
{${n>1}$ is odd and ${k>0}$;} here~$g$ twists $1$ or $2$ fibres, which are the fibres twisted by $g^n$ that are invariant by~$g$.
\item[\upshape 4.]
{ ${n}$ is even and ${k>0}$;} here~$g$ twists $r=1$ or $2$ singular fibres; none of them are twisted by $g^n$; moreover the action of~$g$ on the set of $2k$ fibres twisted by $g^n$ is fixed-point free; furthermore, $n$~divides $2k$, and $2k/n\equiv r \pmod{2}$.
\end{itemize}
\end{prop}
\begin{proof}
Lemma~\ref{Lem:DeJI} describes the situation when $n=1$. We now assume that $n>1$; by blowing-down some components of singular fibres we may also suppose that the triple $(G,S,\pi)$ is minimal. 

Denote by $a_1,a_2\in\mathbb{P}^1$ the two points fixed by $\overline{\pi}(g)\in\Aut(\mathbb{P}^1)$. For $i\not\equiv0\pmod{n}$ the element $\overline{\pi}(g^i)$ fixes only two points of $\mathbb{P}^1$, namely $a_1$ and $a_2$ (since $\overline{\pi}(g)$ has order~$n$); the only possible fibres twisted by $g^i$ are thus $\pi^{-1}(a_1),\pi^{-1}(a_2)$.

Suppose that $g^n$ does not twist any singular fibre. By minimality 
there are at most $2$ singular fibres ($\pi^{-1}(a_1)$ and/or $\pi^{-1}(a_2)$) of $\pi$ and~$g$ twists each one. Lemma~\ref{Lem:Degree567CB} tells us that $S$ is a del Pezzo surface of degree $6$ and Lemma~\ref{Lem:KappaDP6} shows that
\begin{center}
$\begin{array}{lllll}
g=\kappa_{\alpha,\beta}:&\h{3}\big((x:y:z) , (u:v:w)\big) \mapsto \big((u:\alpha w:\beta v)&\h{3}, &\h{3}(x:\alpha^{-1} z:\beta^{-1} y)\big),\end{array}$\end{center} for some $\alpha,\beta \in \K^{*}$. We compute the square of~$g$ and find
\begin{center}
$\begin{array}{lllll}
g^2:&\h{3}\big((x:y:z) , (u:v:w)\big) \mapsto \big((x:\alpha\beta^{-1} y:\alpha^{-1}\beta z)&\h{3}, &\h{3}(u:\alpha^{-1}\beta v:\alpha\beta^{-1} w)\big).\end{array}$\end{center}
Consequently, the order of~$g$ is $2n$. The fact that $g^i$ twists $\pi^{-1}(a_1)$ and $\pi^{-1}(a_2)$ when $i$ is odd implies that~$n$ is even. Case $2$ is complete.

If $g^n$ twists at least one singular fibre, it twists an even number of singular fibres (Lemma \ref{Lem:DeJI}) which we denote by $2k$, and $g^n$ is an involution. If~$n$ is odd, each fibre twisted by $g^n$ is twisted by~$g$, and conversely; this yields case $3$. It remains to consider the more difficult case when~$n$ is even.

Firstly we observe that there are $r+2k$ singular fibres with $r\in\{1,2\}$, corresponding to the points $a_1$ and/or $a_2$, $c_1,...,c_{2k}$ of $\mathbb{P}^1$, the first $r$ of them being twisted by~$g$ and the $2k$ others by $g^n$. Under the permutation $\overline{\pi}(g)$, the set $\{c_1,...,c_{2k}\}$ decomposes into disjoint cycles of length~$n$ (this action is fixed-point-free); this shows that~$n$ divides $2k$. We write $t=2k/n\in\mathbb{N}$ and set $\{c_1,...,c_{2k}\}=\cup_{i=1}^t C_i$, where each $C_i\subset \mathbb{P}^1$ is an orbit of $\overline{\pi}(g)$ of size~$n$. To deduce the congruence $r\equiv t\pmod{2}$, we study the action of~$g$ on $\Pic{S}$. 

 For $i\in\{1,...,t\}$, choose $F_{i}$ to be a component in the fibre of the singular fibre of some point of $C_i$, and for $i\in\{1,r\}$ choose $L_i$ to be a component in the fibre of $a_i$. Let us write \begin{center}$R=\sum_{i=1}^t (F_i+g(F_i)+...+g^{n-1}(F_i))+\sum_{i=1}^r L_i\in\Pic{S}.$\end{center}
 
 Denoting by $f\subset S$ a general fibre of $\pi$, we find the equalities $g(L_i)=f-L_i$ and $g^n(F_i)=f-F_i$ in $\Pic{S}$, which yield  (once again in $\Pic{S}$):
 \begin{center}$g(R)=R+(r+t)f-2(\sum_{i=1}^r L_i +\sum_{i=1}^t F_i).$\end{center}
 
 The contraction of the divisor $R$ gives rise to a birational morphism of conic bundles (not $g$-equivariant) $\nu:S\rightarrow \mathbb{F}_m$ for some integer $m\geq 0$. Denote by $s\subset S$ the pull-back by $\nu$ of a general section of $\mathbb{F}_m$ of self-intersection~$m$ (which does not pass through any of the base-points of $\nu^{-1}$). The canonical divisor $K_S$ of $S$ is then equal in $\Pic{S}$ to the divisor $-2s+(m-2)f+R$. We compute $g(2s)$ and $2(g(s)-s)=g(2s)-2s$ in $\Pic{S}$:
 \begin{center}$\begin{array}{rcl}g(2s)&=&g(-K_S+(m-2)f+R)=-K_S+(m-2)f+g(R);\\
g(2s)-2s&=&g(R)-R=(r+t)f-2(\sum_{i=1}^r L_i +\sum_{i=1}^t F_i).\end{array}$\end{center}
 This shows that $(r+t)f\in 2\Pic{S}$, which implies that $r\equiv t\pmod{2}$. Case $4$ is complete.\end{proof}

\begin{coro}\label{Cor:RootRatTwist}
If $g\in \Aut(S,\pi)$ is a root of a twisting involution~$h$ that fixes a rational curve (i.e.\ that twists $2$ singular fibres) and if~$g$ twists at least one fibre not twisted by~$h$, then $g^2=h$,~$g$ twists exactly one singular fibre, and it exchanges the two fibres twisted by~$h$.
\end{coro}
\begin{proof}
We apply Proposition~\ref{Prp:DescriptionTwistingElementsFinite} and obtain case $4$ with $k=1$.
\end{proof}

Corollary \ref{Cor:RootRatTwist} and the following result will be useful in the sequel.
\begin{lemm}
\label{Lem:CommutLeav}
Let $g\in \Aut(S,\pi)$ be a non-trivial automorphism of finite order that leaves every component of every singular fibre of $\pi$ invariant (i.e.\ that acts trivially on $\Pic{S}$) and let $h\in\Aut(S,\pi)$ be an element that commutes with~$g$. Then, either no singular fibre of $\pi$ is twisted by~$h$ or each singular fibre of $\pi$ which is invariant by~$h$ is twisted by~$h$.
\end{lemm}
\begin{proof}
If no twisting element belongs to $\Aut(S,\pi)$, we are done. Otherwise, the birational morphism of conic bundles $p_0:S\rightarrow \mathbb{P}^1\times\mathbb{P}^1$ given by Lemma~\ref{Lem:GoingToF0F1} conjugates~$g$ to an element of finite order of $\Aut(\mathbb{P}^1\times\mathbb{P}^1,\pi_1)$ whose set of fixed points is the union of two rational curves. The set of points of $S$ fixed by~$g$ is thus the union of two sections and a finite number of points (which are the singular points of the singular fibres of $\pi$). Any element $h\in \Aut(S,\pi)$ that commutes with~$g$ leaves the set of these two sections invariant. More precisely, the action on one invariant singular fibre $F$ implies the action on the two sections:~$h$ exchanges the two sections if and only if it twists $F$. Since the situation is the same at any other singular fibre, we obtain the result.
\end{proof}

We conclude this section with some results on automorphisms of infinite order of conic bundles, which will not help us directly here but seem interesting to observe.

\begin{prop}[Classification of twisting elements of infinite order]
\label{Prp:DescriptionTwistingElementsInfinite}
Let $(S,\pi)$ be a conic bundle and $g \in \Aut(S,\pi)$ be a twisting automorphism of infinite order.

Then~$g$ twists exactly two fibres of $\pi$ and there exists some $g$-equivariant birational morphism of conic bundles $\eta:S\rightarrow S_6$, where $S_6$ is the del Pezzo surface of degree $6$ and $\eta g\eta^{-1}=\kappa_{\alpha,\beta}$ for some $\alpha,\beta \in \K^{*}$. 
\end{prop}
\begin{proof}
Assume that the triple $(<g>,S,\pi)$ is minimal. Lemma \ref{Lem:DeJI} shows that no twisting element of infinite order acts trivially on the basis of the fibration. Consequently, $g^k$ acts trivially on the basis if and only if $k=0$, whence $g^k$ twists a fibre $F$ if and only if $k$ is odd and~$g$ twists $F$. There thus exist at most $2$ singular fibres of $\pi$, and Lemma~\ref{Lem:Degree567CB} tells us that $S$ is a del Pezzo surface of degree $6$. Lemma~\ref{Lem:KappaDP6} shows that $g=\kappa_{\alpha,\beta}$ for some $\alpha,\beta \in \K^{*}$.
\end{proof}

\begin{coro}
Let $g\in \Aut(S,\pi)$ be an element of infinite order; then a birational morphism conjugates~$g$ to an automorphism of a Hirzebruch surface.
\end{coro}
\begin{proof}
Assume that the triple $(<g>,S,\pi)$ is minimal. If the fibration is smooth, we are done. Otherwise, a birational morphism conjugates~$g$ to an automorphism $\kappa_{\alpha,\beta}\in\Aut(S_6)$ of a conic bundle on the del Pezzo surface of degree $6$ (Lemma~\ref{Prp:DescriptionTwistingElementsInfinite}). We conclude by using Lemma~\ref{Lem:DP6CBPasMin}.
\end{proof}
\section{The example $\Cs{24}$}
\label{Sec:ExampleCs24} 
We now give the most important example of this paper. This is the only finite Abelian subgroup of the Cremona group which is not conjugate to a group of automorphisms of\hspace{0.2 cm}$\Pn$ or $\mathbb{P}^1\times\mathbb{P}^1$ but whose non-trivial elements do not fix any curve of positive genus (Theorem~\ref{Thm:NonCyclic}).

\bigskip

Let $S_6 \subset \Pn \times \Pn$ be the del Pezzo surface of degree $6$ (see Section~\ref{Sec:DelPezzo6}) defined by 
\begin{center}$S_6=\{ \big((x:y:z) , (u:v:w)\big)\ | \ ux=yv=zw\};$\end{center} we keep the notation of Section~\ref{Sec:DelPezzo6}. We denote by $\eta:\hat{S_4}\rightarrow S_6$ the blow-up of $A_4,A_5 \in S_6$ defined by
\begin{center}$\begin{array}{ccccccc}A_4=&\big(&(0:1:1)&,&(1:0:0)&\big)& \in D_{23}, \vspace{0.1 cm}\\
A_5=&\big(&(1:0:0)&,&(0:1:-1)&\big)& \in E_1.\end{array}$\end{center}

We again denote by $E_1,E_2,E_3,D_{12},D_{13},D_{23}$ the total pull-backs by $\eta$ of these divisors of $S_6$. We denote by $\widetilde{E_1}$ and $\widetilde{D_{23}}$ the strict pull-backs of $E_1$ and $D_{23}$ by $\eta$. (Note that for the other exceptional divisors, the strict and total pull-backs are the same.) Let us illustrate the situations on the surfaces $S_6$ and $\hat{S_4}$ respectively:
\begin{center}
\includegraphics[width=90.00mm]{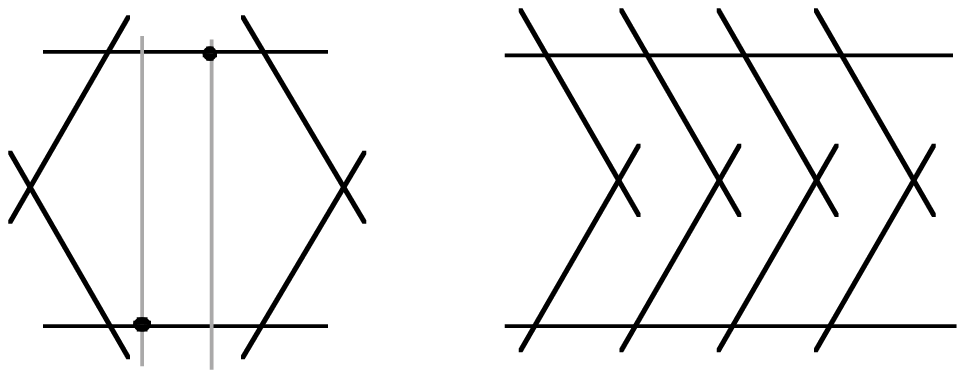}%
\drawat{-63.78mm}{30mm}{$D_{23}$}%
\drawat{-60.8mm}{23.25mm}{$E_3$}%
\drawat{-60.93mm}{8.12mm}{$D_{13}$}%
\drawat{-64.07mm}{1.0mm}{$E_1$}%
\drawat{-88mm}{8.12mm}{$D_{12}$}%
\drawat{-87.0mm}{23.25mm}{$E_2$}%
\drawat{-74.47mm}{6.16mm}{$A_5$}%
\drawat{-68.49mm}{25.03mm}{$A_4$}\drawat{-8.94mm}{30mm}{$\widetilde{D_{23}}$}%
\drawat{-8mm}{0.5mm}{$\widetilde{E_1}$}%
\drawat{-35.66mm}{22.6mm}{$E_2$}%
\drawat{-26.96mm}{22.6mm}{$D_{15}$}%
\drawat{-18.14mm}{22.6mm}{$E_4$}%
\drawat{-9.55mm}{22.6mm}{$E_3$}%
\drawat{-38mm}{6.9mm}{$D_{12}$}%
\drawat{-28.8mm}{6.9mm}{$E_5$}%
\drawat{-21mm}{6.9mm}{$D_{14}$}%
\drawat{-11.9mm}{6.9mm}{$D_{13}$}\end{center}

Let $\pi_1$ denote the morphism $S_6 \rightarrow \mathbb{P}^1$ defined in Section~\ref{Sec:DelPezzo6}. The morphism $\pi=\pi_1 \circ \eta$ gives the surface $\hat{S_4}$ a conic bundle structure $(\hat{S_4},\pi)$. It has $4$ singular fibres, which are the fibres of $(-1:1)$, $(0:1)$, $(1:1)$ and $(1:0)$.
We denote by $f$ the divisor of $\hat{S_4}$ corresponding to a fibre of $\pi$ and set $E_4=\eta^{-1}(A_4)$, $E_5=\eta^{-1}(A_5)$. Note that $E_4$ is one of the components of the singular fibre of $(1:1)$; we denote by $D_{14}=f-E_4$ the other component, which is the strict pull-back by $\eta$ of $\pi_1^{-1}(1:1)$. Similarly, we denote by $D_{15}$ the divisor $f-E_5$, so that the singular fibre of $(-1:1)$ is $\{E_5,D_{15}\}$.

\begin{lemm}\label{Lem:10irreduciblecurves}
On the surface $\hat{S_4}$ there are exactly $10$ irreducible rational smooth curves of negative self-intersection. Explicitly, the $8$ curves \begin{center}$E_2,E_3,E_4,E_5,D_{12},D_{13},D_{14},D_{15}$\end{center}
have self-intersection $-1$ and the two curves \begin{center}$\widetilde{E_1}=E_1-E_5$ and $\widetilde{D_{23}}=D_{23}-E_4$\end{center} have self-intersection $-2$.
 \end{lemm}
 \begin{proof}
The difficult part is to show that every rational irreducible smooth curve of negative self-intersection is one of the ten given above. Let $C$ be such a curve. 

Denote by $L$ the pull-back of a general line of $\Pn$ by the blow-up $\pr_1\circ\eta:\hat{S_4}\rightarrow \Pn$ of the five points. If $C$ is collapsed by $\pr_1\circ\eta$, then $C$ is one of the curves $\widetilde{E_1},E_2,E_3,E_4,E_5$. Otherwise, $C=mL-\sum_{i=1}^5 a_i E_i$, where $m,a_1,...,a_5$ are non-negative integers, and $m>0$. Since $C$ is rational we have $C\cdot(C+K_{\hat{S_4}})=-2$, and by hypothesis $C^2=-r$ for some positive integer $r$. The relations $C^2=-r$ and $C\cdot K_{\hat{S_4}}=r-2$ imply (since $K_{\hat{S_4}}=-3L+\sum_{i=1}^5 E_i$) the equations
\begin{equation}\label{eq:Arith}\begin{array}{rcl}
\sum_{i=1}^5 a_i^2&=&m^2+r, \vspace{0.1 cm}\\
\sum_{i=1}^5 a_i& =& 3m+r-2.\end{array}\end{equation}
If $m=r=1$, we find that $C$ is the pull-back of a line passing through two of the points, so $C=D_{1i}$ for some $i\in\{2,...,5\}$. If $m=2$ and $r=1$, $C$ is the pull-back of a conic passing through each blown-up point.
The configuration of the points eliminates this possibility. If $m=1$ and $r=2$, we obtain a line passing through three blown-up points, so $C=\widetilde{D_{23}}$. 

We now prove that if there is no integral solution to (\ref{eq:Arith}) for $m,r\geq 2$.
Since $(\sum_{i=1}^5 a_i)^2\leq 5(\sum_{i=1}^5 a_i^2)$ (by the Cauchy-Schwarz inequality with the vectors $(1,...,1)$ and $(a_1,...,a_5)$), we obtain $(3m+(r-2))^2\leq 5(m^2+r)$, and this gives
\begin{center}$\begin{array}{rcl}
4m^2-10+(r-2)\cdot(6 m+r-7)&\leq& 0.\\
\end{array}$\end{center}
But this is not possible if $m,r\geq 2$, since in this case $4m^2>10$ and $6m+r>7$.
 \end{proof}
Note that $(K_{\hat{S_4}})^2=4$, which is why we denote this surface by $\hat{S_4}$; the hat is here because the surface is not a del Pezzo surface, since it contains irreducible divisors of self-intersection $-2$. 
 \begin{coro}\label{Cor:S4onefib}
 There is only one conic bundle structure on $\hat{S_4}$, which is the one induced by $\pi=\pi_1\circ \eta$.
 \end{coro}
 \begin{proof}
 Since $(K_{\hat{S_4}})^2=4$, the number of singular fibres of any conic bundle is $4$, and thus it consists of eight $(-1)$-curves $C_1,...,C_8$. The divisor of a fibre of the conic bundle is thus equal to $\frac{1}{4}\sum_{i=1}^8 C_i$. Since there are exactly eight $(-1)$-curves on $\hat{S_4}$, there is no choice.
 \end{proof}

The group of automorphisms of $\hat{S_4}$ that leave every curve of negative self-intersection invariant is isomorphic to $\K^{*}$ and corresponds to automorphisms of\hspace{0.2 cm}$\Pn$ of the form $(x:y:z)\mapsto (\alpha x:y:z)$, for $\alpha \in \K^{*}$. Indeed, such automorphisms are the lifts of automorphisms of $S_6$ leaving invariant every exceptional curve (which are of the form $\big((x:y:z) , (u:v:w)\big) \mapsto \big((x:\alpha y:\beta z), (u:\alpha^{-1} v:\beta^{-1} w)\big)$, for $\alpha,\beta \in \K^{*}$) and which fix both points $A_4$ and $A_5$. 

\begin{defi}
\label{Def:Cs24}
Let $h_1$ and $h_2$ be the following birational transformations of $\mathbb{P}^2$:
\begin{center}\begin{tabular}{l}
$h_1:(x:y:z)\dasharrow (yz:xy:-xz)$ \\
$h_2:(x:y:z)\dasharrow ( yz(y-z):xz(y+z):xy(y+z))$ \end{tabular}\end{center}
and denote respectively by $g_1$, $g_2$ the lift of these elements on $\hat{S_4}$ and by $\Cs{24}$ the group generated by $g_1$ and $g_2$. 
\end{defi}
The following lemma shows that $\Cs{24}\subset\Aut(\hat{S_4},\pi)$ and describes some of the properties of this group.

\begin{lemm}
\label{Lem:Cs24Properties}
Let $h_1,h_2, g_1, g_2, \Cs{24}$ be as in Definition~$\ref{Def:Cs24}$. Then:
\begin{enumerate}
\item[\upshape 1.]
The group $\Cs{24}$ is a group of automorphisms of $\hat{S_4}$ that preserve the conic bundle $(\hat{S_4},\pi)$, i.e.\ $\Cs{24}\subset \Aut(\hat{S_4},\pi)$.
\item[\upshape 2.]
The action of $g_1$ and $g_2$ on the set of irreducible rational curves of negative self-intersection is respectively:
\begin{center}$(\widetilde{E_1}\ \widetilde{D_{23}})(E_2\ D_{12})(E_3\ D_{13})(E_4\ E_5)(D_{14}\ D_{15})$, \vspace{0.1 cm}\\
$(\widetilde{E_1}\ \widetilde{D_{23}})(E_2\ D_{13})(E_3\ D_{12})(E_4\ D_{14})(E_5\ D_{15})$.\end{center}
In particular, both $g_1$ and $g_2$ twist the conic bundle $(\hat{S_4},\pi)$.
\item[\upshape 3.]
Both $g_1$ and $g_2$ are elements of order $4$ and \begin{center}$(h_1)^2=(h_2)^2=(x:y:z)\mapsto (-x:y:z)$.\end{center} Thus $(g_1)^2=(g_2)^2 \in \ker \overline{\pi}$ is an automorphism of $\hat{S_4}$ which leaves every divisor of negative self-intersection invariant.
\item[\upshape 4.]
The group $\Cs{24}$ is isomorphic to $\Z{2}\times\Z{4}$ and the action on the basis of the fibration $\pi$ yields the exact sequence
\begin{center}$1\rightarrow<(h_1)^2>\cong\Z{2}\rightarrow \Cs{24}\stackrel{\overline{\pi}}{\rightarrow} <\overline{\pi}(h_1),\overline{\pi}(h_2)>\cong (\Z{2})^2\rightarrow 1.$\end{center}
\item[\upshape 5.]
The group $\Cs{24}$ contains no involution that twists the conic bundle $(\hat{S_4},\pi)$. In particular, no element of $\Cs{24}$ fixes a curve of positive genus.
\item[\upshape 6.]
The pair $(\Cs{24},\hat{S_4})$ and the triple $(\Cs{24},\hat{S_4},\pi)$ are both minimal.
\end{enumerate}\end{lemm}
\begin{proof}\upshape
Observe first that $h_1$ and $h_2$ preserve the pencil of lines of\hspace{0.2 cm}$\Pn$ passing through the point $A_1=(1:0:0)$, so $g_1,g_2$ are birational transformations of $\hat{S_4}$ that send a general fibre of $\pi$ on another fibre. Then, we compute $(h_1)^2=(h_2)^2=(x:y:z)\mapsto (-x:y:z)$. This implies that both $h_1$ and $h_2$ are birational maps of order $4$.

Note that the lift of $h_1$ on the surface $S_6$ is the automorphism 
\begin{center}
$\kappa_{1,-1}:\big((x:y:z) , (u:v:w)\big) \mapsto \big((u: w: -v),(x: z:-y)\big)$\end{center}
(see Example~\ref{Exa:KappaAB}). Since this automorphism permutes $A_4$ and $A_5$, its lift on $\hat{S_4}$ is biregular. The action on the divisors with negative self-intersection is deduced from that of $\kappa_{1,-1}$ (see Lemma~\ref{Lem:KappaDP6}).

Compute the involution \begin{center}$h_3=h_1h_2=(x:y:z) \dasharrow (x(y+z):z(y-z):-y(y-z)).$\end{center} Its linear system is \begin{center}$\{ax(y+z)+(by+cz)(y-z)=0\ | \ (a:b:c) \in \Pn\},$\end{center} which is the linear sytem of conics passing through $(0:1:1)$ and $A_1=(1:0:0)$, with tangent $y+z=0$ at this point (i.e.\ passing through $A_5$). Blowing-up these three points (two on $\Pn$ and one in the blow-up of $A_1$), we get an automorphism $g_3$ of some rational surface. As the points $A_2=(0:1:0)$ and $A_3=(0:0:1)$ are permuted by $h_3$, we can also blow them up and again get an automorphism.
The isomorphism class of the surface obtained is independent of the order of the blown-up points. We may first blow-up $A_1,A_2,A_3$ and get $S_6$. Then, we blow-up the two other base-points of $h_3$, which are in fact $A_4$ (the point $(0:1:-1)$) and $A_5$ (the point infinitely near to $A_1$ corresponding to the tangent $y+z=0$). This shows that $g_3$, and therefore $g_2$, belong to $\Aut(\hat{S_4},\pi)$.

Since $h_3$ permutes the points $A_2$ and $A_3$, $g_3=g_1g_2$ permutes the divisors $E_2$ and $E_3$. It also permutes $D_{12}$ and $D_{13}$, since $h_3$ leaves the pencil of lines passing through $A_1$ invariant. It therefore leaves $\widetilde{E_1}$ and $\widetilde{D_{23}}$ invariant, since $E_2$ and $E_3$ touch $\widetilde{D_{23}}$ but not $E_1$. The remaining exceptional divisors are $E_4,E_5,D_{14},D_{15}$. Either $g_1g_2$ leaves all four invariant, or it acts as $(E_4\ D_{15})(E_5\ D_{14})$ (using the intersection with $\widetilde{E_1}$ and $\widetilde{D_{23}}$). Since $A_4$ and $A_5$ are base-points of $h_1h_2$, $E_4$ and $E_5$ are not invariant. Thus, $g_1g_2$ acts on the irreducible rational curves of negative self-intersection as
$(E_2\ E_3)(D_{12}\ D_{13})(E_4\ D_{15})(E_5\ D_{14})$.
We obtain the action of $g_2$ by composing that of $g_1g_2$ with that of $g_1$ and thus have proved assertions~$1$ through~$3$.

Assertion $4$ follows from assertion $3$ and the fact that $g_1$ and $g_2$ commute.

Let us prove that $\Cs{24}$ contains no involution that twists the conic bundle $(\hat{S_4},\pi)$. Recall that such elements are involutions acting trivially on the basis of the fibration (see Lemma~\ref{Lem:DeJI}). Note that the $2$-torsion of $\Cs{24}$ is equal to $\{1,g_1^2,g_1g_2,g_1g_2^{-1}\}$. The elements $g_1g_2$ and $g_1g_2^{-1}$ do not act trivially on the basis of the fibration, and the element $(g_1)^2$ does not twist any singular fibre since it leaves every curve of negative self-intersection invariant. This proves assertion $5$.

It remains to prove the last assertion. Observe that the orbits of the action of $\Cs{24}$ on the exceptional divisors of $\hat{S_4}$ are $\{E_2,E_3,D_{12},D_{13}\}$ and $\{E_4,E_5,D_{14},D_{15}\}$. Since these orbits cannot be contracted, the pair $(\Cs{24},\hat{S_4})$ is minimal, and so is the triple $(\Cs{24},\hat{S_4},\pi)$. 
\end{proof}

\begin{rema}
The pair $(\Cs{24},\hat{S_4})$ was introduced in \cite{bib:JBTh} and was called $\nump{Cs.24}$ because it is a group acting on a {\bf c}onic bundle, which is {\bf s}pecial, and isomorphic to $\Z{2}\times\Z{4}$.
\end{rema}

\section{Finite Abelian groups of automorphisms of conic bundles - birational representative elements}
\label{Sec:ConBundleFinal}
In this section we use the tools prepared in the previous sections
to describe the finite Abelian groups of automorphisms of conic bundles such that \emph{no non-trivial element fixes a curve of positive genus}.

We first treat the case in which no involution twisting the conic bundle belongs to the group:
\begin{prop}
\label{Prp:necDeJi}
Let $G \subset \Aut(S,\pi)$ be a finite Abelian group of automorphisms of the conic bundle $(S,\pi)$ such that:
\begin{itemize}
\item
no involution that twists the conic bundle $(S,\pi)$ belongs to $G$;
\item
the triple $(G,S,\pi)$ is minimal.
\end{itemize}
Then, one of the following occurs:
\begin{itemize}
\item
The fibration is smooth, i.e.\ $S$ is a Hirzebruch surface.
\item
$S$ is the del Pezzo surface of degree $6$.
\item 
The triple $(G,S,\pi)$ is isomorphic to the triple $(\Cs{24},\hat{S_4},\pi)$ of Section~$\ref{Sec:ExampleCs24}$.
\end{itemize}
\end{prop}
\begin{proof}
We assume that the fibration is not smooth. 
Recall that since the triple $(G,S,\pi)$ is minimal, any singular fibre of $\pi$ is twisted by an element of $G$ (by Lemma~\ref{Lem:MinTripl}). Since no twisting involution belongs to $G$, any element~$g\in G$ that twists a fibre corresponds to case~$2$ of Proposition~\ref{Prp:DescriptionTwistingElementsFinite}. In particular,~$g$ is the lift on $S$ of an automorphism of the form $\kappa_{\alpha,\beta}$ of the del Pezzo surface of degree $6$ and it twists $2$ singular fibres, which correspond to the fibres of the two fixed points of $\overline{\pi}(g)\in \PGL(2,\K)$. Furthermore,~$g$ is the root of an involution that leaves every component of every singular fibre of $\pi$ invariant.

If the number of singular fibres is exactly two, then $S$ is the del Pezzo surface of degree $6$, and we are done.

Now suppose that the number of singular fibres is larger than two. This implies that $\overline{\pi}(G)$ is not a cyclic group (otherwise the non-trivial elements of $\overline{\pi}(G)$ would have the same two fixed points: there would then be at most two singular fibres); therefore, $\overline{\pi}(G)$ is isomorphic to $(\Z{2})^2$. By a judicious choice of coordinates we may suppose that \begin{center}$\overline{\pi}(G)=\left\{\left(\begin{array}{cc}1 & 0\\ 0 & 1\end{array}\right),\left(\begin{array}{cc}-1 & 0\\ 0 & 1\end{array}\right),\left(\begin{array}{cc}0 & 1\\ 1 & 0\end{array}\right),\left(\begin{array}{cc}0 & -1\\ 1 & 0\end{array}\right)\right\}.$\end{center}

Since a singular fibre corresponds to a fixed point of one of the three elements of order $2$ of $\overline{\pi}(G)$, only the fibres of $(0:1),(1:0),(1:1),(-1:1),(\im:1),(-\im:1)$ can be singular. Since the group $\overline{\pi}(G)$ acts transitively on the sets $\{(1:0),(0:1)\}$, $\{(1:\pm 1)\}$ and $\{(1:\pm \im)\}$, there are $4$ or $6$ singular fibres.

We denote by $g_1$ an element of $G$ which twists the two singular fibres of $(1:0)$ and $(0:1)$. 
Let $\eta:S\rightarrow S_6$ denote the birational $g_1$-equivariant morphism given by Proposition~\ref{Prp:DescriptionTwistingElementsFinite}, which conjugates $g_1$ to the automorphism
\begin{center}$\eta g_1 \eta^{-1}=\kappa_{\alpha,\beta}:\big((x:y:z) , (u:v:w)\big) \mapsto \big((u:\alpha w:\beta v), (x:\alpha^{-1} z:\beta^{-1} y)\big)$\end{center}
of the del Pezzo surface $S_6$ of degree $6$, for some $\alpha,\beta \in \K^{*}$. In fact, since $\overline{\pi}(g_1)$ has order $2$, we have $\beta=-\alpha$, so $\eta g_1 \eta^{-1}=\kappa_{\alpha,-\alpha}$.
The points blown-up by $\eta$ are fixed by 
\begin{center}$\eta (g_1)^2 \eta^{-1}=(\kappa_{\alpha,-\alpha})^2:\big((x:y:z) , (u:v:w)\big)\mapsto \big((x:-y:- z),(u:-v:-w)\big)$, \end{center}
and therefore belong to the curves
\begin{center} $\begin{array}{rlll}
& E_1&=& \{\big((1:0:0), (0:a:b)\big) \ | \ (a:b) \in \mathbb{P}^1\} \\
\mbox{and} &D_{23}&=&\{\big((0:a:b) , (1:0:0)\big) \ | \ (a:b) \in \mathbb{P}^1\}.\end{array}$\end{center} Since these points consist of orbits of $\eta g_1\eta^{-1}$, half of them lie in $E_1$ and the other half in $D_{23}$. In fact, up to a change of coordinates, $\big((x,y,z),(u,v,w)\big) \leftrightarrow \big((u,v,w),(x,y,z)\big)$, the points that may be blown-up by $\eta$ are
\begin{center}
$\begin{array}{rllllll}
A_4=&\big(\h{3}&(0:1:1) &\h{2},\h{2}& (1:0:0)& \h{3}\big)\in D_{23},\\
\kappa_{\alpha,-\alpha}(A_4)=A_5=&\big(\h{3}&(1:0:0) &\h{2},\h{2}& (0:1:-1)& \h{3}\big)\in E_1,\\
A_6=&\big(\h{3}&(0:1:\im) &\h{2},\h{2}&(1:0:0)& \h{3}\big)\in D_{23},\\
\kappa_{\alpha,-\alpha}(A_6)=A_7=&\big(\h{3}&(1:0:0) &\h{2},\h{2}& (0:1:\im)& \h{3}\big)\in E_1.\end{array}$
\end{center}
The strict pull-backs $\widetilde{E_1}$ and $\widetilde{D_{23}}$ by $\eta$ of $E_1$ and $D_{23}$ respectively thus have self-intersection $-2$ or $-3$ in $S$, depending on the number of points blown-up.
By convention we again denote by $E_1,E_2,E_3,D_{12},D_{13},D_{23}$ the total pull-backs by $\eta$ of these divisors. (Note that for $E_2,E_3,D_{12},D_{13}$, the strict and the total pull-backs are the same.)
We set $E_4=\eta^{-1}(A_4)$,..., $E_7=\eta^{-1}(A_7)$ and denote by $f$ the divisor class of the fibre of the conic bundle. 

{ (a)}
Suppose that $\eta$ is the blow-up of $A_4$ and $A_5$, which implies that $S$ is the surface $\hat{S_4}$ of Section~\ref{Sec:ExampleCs24}. 
The Picard group of $S$ is then generated by $E_1,E_2,...,E_5$ and $f$.

Since we assumed that $(G,S,\pi)$ is minimal, the singular fibres of $(1:1)$ and $(-1:1)$ must be twisted. 
One element $g_2$ twists these two singular fibres and acts with order $2$ on the basis of the fibration, with action $(x_1:x_2) \mapsto (x_2:x_1)$. Since $g_1$ and $g_2$ twist some singular fibre, both must invert the two curves of self-intersection $-2$, namely $\widetilde{E_1}$ and $\widetilde{D_{23}}$. The action of $g_1$ and $g_2$ on the irreducible rational curves of negative self-intersection is then respectively
\begin{center}$\begin{array}{l}(\widetilde{E_1}\ \widetilde{D_{23}})(E_2\ D_{12})(E_3\ D_{13})(E_4\ E_5)(D_{14}\ D_{15}), \vspace{0.1 cm}\\
(\widetilde{E_1}\ \widetilde{D_{23}})(E_2\ D_{13})(E_3\ D_{12})(E_4\ D_{14})(E_5\ D_{15}).\end{array}$\end{center}
The elements $g_1$ and $g_2$ thus have the same action on $\Pic{S}=\Pic{\hat{S_4}}$ as the two automorphisms with the same name in Definition~\ref{Def:Cs24} and Lemma~\ref{Lem:Cs24Properties}, which generate $\Cs{24}$. Note that the group $H$ of automorphisms of $S$ that leave every curve of negative self-intersection invariant is isomorphic to $\K^{*}$ and corresponds to automorphisms of\hspace{0.2 cm}$\Pn$ of the form $(x:y:z)\mapsto (\alpha x:y:z)$, for any $\alpha \in \K^{*}$. Then, $g_1$ and $g_2$ are equal to the lift of the the following birational maps of $\Pn$:
\begin{center}$\begin{array}{l}h_1:(x:y:z)\dasharrow (\mu yz:xy:-xz),\\
h_2:(x:y:z)\dasharrow (\nu yz(y-z):xz(y+z):xy(y+z)),\end{array}$\end{center} 
for some $\mu,\nu \in \K^{*}$.

As $h_1h_2(x:y:z)=(\mu x(y+z):\nu z(y-z):-\nu y(y-z))$ and $h_2h_1(x:y:z)=(\nu x(y+z):\mu z(y-z):-\mu y(y-z))$ must be the same by hypothesis, we get $\mu^2=\nu^2$.

We observe that $\overline{\pi}(g_1)$ and $\overline{\pi}(g_2)$ generate $\overline{\pi}(G)\cong (\Z{2})^2$; on the other hand, by hypothesis an element of $G'$ does not twist a singular fibre and hence belongs to $H$. As the only elements of $H$ which commute with $g_1$ are $id$ and $(g_1)^2$ (which is the lift of $(h_1)^2:(x:y:z)\mapsto (-x:y:z)$), we see that $g_1$ and $g_2$ generate the whole group $G$.

Conjugating $h_1$ and $h_2$ by $(x:y:z) \mapsto (\alpha x: y:z)$, where $\alpha \in \K^{*}, \alpha^2=\mu$, we may suppose that $\mu=1$. So $\nu=\pm 1$ and we get in both cases the same group, because $(h_1)^2(x:y:z)=(-x:y:z)$. The triple $(G,S,\pi)$ is hence isomorphic to the triple $(\Cs{24},\hat{S_4},\pi)$ of Section~\ref{Sec:ExampleCs24}.

{ (b)}
Suppose that $\eta$ is the blow-up of $A_6$ and $A_7$. 
We get a case isomorphic to the previous one, using the automorphism $\big((x:y:z) , (u:v:w)\big) \mapsto \big((x:y:\im z), (u:v:-\im w)\big)$ of $S_6$.

{ (c)}
Suppose that $\eta$ is the blow-up of $A_4,A_5,A_6$ and $A_7$. 
The Picard group of $S$ is then generated by $E_1,E_2,...,E_6,E_7$ and $f$.
Since $(G,S,\pi)$ is minimal, there must be two elements $g_2,g_3 \in G$ that twist respectively the fibres of $(\pm 1:1)$ and those of $(\pm \im:1)$.
As in the previous example, the three actions of these elements on the basis are of order $2$, and the three elements transpose $\widetilde{E_1}$ and $\widetilde{D_{23}}$. The actions of $g_1,g_2$ and $g_3$ on the set of irreducible components of the singular fibres of $\pi$ are then respectively
\begin{center}
$\begin{array}{l}(E_2\ D_{12})(E_3\ D_{13})(E_4\ E_5)(D_{14}\ D_{15})(E_6\ E_7)(D_{16}\ D_{17}), \vspace{0.1 cm}\\
(E_2\ D_{13})(E_3\ D_{12})(E_4\ D_{14})(E_5\ D_{15})(E_6\ E_7)(D_{16}\ D_{17}), \vspace{0.1 cm}\\
(E_2\ D_{13})(E_3\ D_{12})(E_4\ E_5)(D_{14}\ D_{15})(E_6\ D_{16})(E_7\ D_{17}).\end{array}$
\end{center} 
This implies that the action of the element $g_1g_2g_3$ is
\begin{center}
$\begin{array}{l}(E_2\ D_{12})(E_3\ D_{13})(E_4 \ D_{14})(E_5 \ D_{15})(E_6\ D_{17})(E_7\ D_{16}),\end{array}$
\end{center}
and thus it twists six singular fibres of the conic bundle and fixes a curve of genus~$2$ (Lemma~\ref{Lem:DeJI}), which contradicts the hypothesis. (In fact, one can also show that the group generated by $g_1$, $g_2$ and $g_3$ is not Abelian, see \cite{bib:JBTh}, page 66.)
\end{proof}
After studying the groups that do not contain a twisting involution, we now study those which contain such elements. Since these twisting involutions cannot fix a curve of positive genus, they twist exactly two fibres (Lemma~\ref{Lem:DeJI}).
\begin{prop}
\label{Prp:HardTwistInv}
Let $G\subset \Aut(S,\pi)$ be a finite Abelian group of automorphisms of a conic bundle $(S,\pi)$ such that:
\begin{enumerate}
\item[\upshape 1.]
If $g\in G$, $g\not=1$, then~$g$ does not fix a curve of positive genus.
\item[\upshape 2.]
The group $G$ contains at least one involution that twists the conic bundle $(S,\pi)$.
\item[\upshape 3.]
The triple $(G,S,\pi)$ is minimal.
\end{enumerate}
Then, $S$ is a del Pezzo surface of degree $5$ or $6$.
\end{prop}
\begin{proof}
If the number of singular fibres is at most $3$, then the surface is a del Pezzo surface of degree $5$ or $6$ (Lemma~\ref{Lem:Degree567CB}). 

We now assume that the number of singular fibres is at least $4$ and show that this situation is not compatible with the hypotheses. We recall once again the exact sequence of Remark~\ref{Rem:ExactSeq}
\begin{equation}
1 \rightarrow G' \rightarrow G \stackrel{\overline{\pi}}{\rightarrow} \overline{\pi}(G) \rightarrow 1,
\tag{\ref{eq:ExactSeqCB}}\end{equation}
and prove the following important assertions:
\begin{itemize}
\item[(a)]
No element of $G$ twists more than two singular fibres.
\item[(b)]
Any twisting involution that belongs to $G$ belongs to $G'$ and twists exactly two singular fibres.
\item[(c)]
Any singular fibre is twisted by an element of $G$.
\item[(d)]
No non-trivial element preserves every component of every singular fibre.
\item[(e)]
Any twisting element of $G$ is a root of (or equal to) a twisting involution that belongs to $G'$.
\end{itemize}

Corollary~\ref{Cor:CurveFixDeJ} shows that an element that twists more than two fibres fixes a curve of positive genus; since this possibility is excluded by hypothesis, we obtain assertion~(a). Lemma~\ref{Lem:DeJI} shows that any twisting involution contained in $G$ belongs to $G'$ and twists an even number of fibres; using assertion~(a), we thus obtain assertion~(b). Assertion~(c) follows from the minimality of the triple $(G,S,\pi)$ (see Lemma~\ref{Lem:MinTripl}). Let us prove assertion (d). Suppose that there exists a non-trivial element $g\in G$ that leaves every component of every singular fibre invariant, and denote by $h\in G'$ a twisting involution (which exists by hypothesis). Since~$g$ and~$h$ commute, Lemma~\ref{Lem:CommutLeav} shows that each singular fibre invariant by~$h$ -- there are at least $4$ -- is twisted by~$h$, which contradicts assertion (a). Therefore, such an element~$g$ doesn't exist and assertion (d) is proved. Finally, Proposition~\ref{Prp:DescriptionTwistingElementsFinite} shows that any twisting element that does not act trivially on the basis of the fibration is a root of an involution that belongs to $G'$, and assertion~(d) shows that this involution is twisting, and we obtain assertion~(e).

\vspace{0.2 cm}

Now that assertions (a) through (e) are proved, we deduce the proposition from them. Let us denote by $\sigma\in G'$ a twisting involution, which twists two singular fibres that we denote by $F_1$ and $F_2$. There are at least two other singular fibres $F_3$ and $F_4$ that are twisted by other elements of $G$.

If $G'=<\sigma>$, the fibres $F_3$ and $F_4$ are twisted by roots of $\sigma$ belonging to $G$ (assertions (c) and (e)). The description of these elements (Proposition~\ref{Prp:DescriptionTwistingElementsFinite}, and in particular Corollary~\ref{Cor:RootRatTwist}) shows that the roots must be square roots that twist exactly one singular fibre and permute the two fibres $F_1$ and $F_2$ twisted by~$\sigma$. There thus exist two elements $h_3,h_4\in G$ that twist respectively the fibres $F_3$ and $F_4$. Since $h_3$ commutes with $h_4$, it must leave invariant the unique fibre twisted by $h_4$, i.e.\ $F_4$. Similarly, $h_4$ must leave $F_3$ invariant. Therefore, $h_3h_4$ leaves the four fibres $F_1$,...,$F_4$ invariant and twists the two fibres $F_3$ and $F_4$; it is thus an involution that belongs to $G'$, which contradicts the fact that $G'=<\sigma>$.

If $G'\not=<\sigma>$, since $\sigma$ has no root in $G'$ (Corollary~\ref{Cor:NoRoot}), the Abelian group $G'\subset \PGL(2,\K(x))$ is isomorphic to $(\Z{2})^2$ and contains (using (d)) three twisting involutions $\sigma$, $\rho$ and $\sigma\rho$. 
Note that two of these three involutions do not twist singular fibres which are all distinct, otherwise the product of the two involutions would give an involution that twists $4$ singular fibres, contradicting (a). We may thus suppose that $\rho$ twists $F_1$ and $F_3$, which implies that $\sigma\rho$ twists $F_2$ and $F_3$.
 The fibre $F_4$ is then twisted by an element which is a square root of one of the three twisting involutions (assertion (e) and Corollary~\ref{Cor:RootRatTwist}). Denote this square root  by~$h$ and suppose that $h^2\not=\sigma$.
 Note that~$h$ exchanges the two singular fibres twisted by $h^2$. One of these is twisted by $\sigma$ and the other is not, so~$h$ and $\sigma$ do not commute.
 \end{proof}

The only remaining possible finite Abelian groups of automorphisms of conic bundles satisfying property $(F)$ are thus del Pezzo surfaces of degree $6$ or~$5$ (studied in Sections~\ref{Sec:DelPezzo6} and~\ref{Sec:DelPezzo5}), the triple $(\Cs{24},\hat{S_4},\pi)$ studied in Section~\ref{Sec:ExampleCs24}, and Hirzebruch surfaces.
We now describe this last case and prove that it is birationally reduced to the case of $\mathbb{P}^1\times\mathbb{P}^1$.
\begin{prop}
\label{Prp:FnAb}
Let $G \subset \Aut(\mathbb{F}_n)$ be a finite Abelian subgroup of automorphisms of $\mathbb{F}_n$, for some integer $n\geq 1$.
 Then, a birational map of conic bundles conjugates $G$ to a finite group of automorphisms of $\mathbb{F}_0=\mathbb{P}^1\times\mathbb{P}^1$ that leaves one ruling invariant.
\end{prop}
\begin{proof}
Let $G \subset \Aut(\mathbb{F}_n)$ be a finite Abelian group, with $n\geq 1$. Note that $G$ preserves the unique ruling of $\mathbb{F}_n$. We denote by $E\subset \mathbb{F}_n$ the unique section
of self-intersection $-n$, which is necessarily invariant by $G$. We have the exact sequence (see Remark~\ref{Rem:ExactSeq})
\begin{equation}
1 \rightarrow G' \rightarrow G \stackrel{\overline{\pi}}{\rightarrow} \overline{\pi}(G) \rightarrow 1.
\tag{\ref{eq:ExactSeqCB}}\end{equation}
Since the group $\overline{\pi}(G) \subset \PGLn{2}$ is Abelian, it is isomorphic to a cyclic group or to $(\Z{2})^2$.

\textit{If $\overline{\pi}(G)$ is a cyclic group,} at least two fibres are invariant by $G$. The group $G$ fixes two points in one such fibre. We can blow-up the point that does not lie on $E$
and blow-down the corresponding fibre to get a group of automorphisms of $\mathbb{F}_{n-1}$. 
We do this~$n$ times and finally obtain a birational map of conic bundles that conjugates $G$ to a group of automorphisms of $\mathbb{F}_0=\mathbb{P}^1 \times
\mathbb{P}^1$.

\textit{If $\overline{\pi}(G)$ is isomorphic to $(\Z{2})^2$,} there exist two fibres $F,F'$ of $\pi$ whose union is invariant by $G$. 
Let $G_F \subset G$ be the subgroup of $G$ of elements that leave $F$ invariant. This group is of index $2$ in $G$ and hence is normal. Since $G_F$ fixes the point $F\cap E$ in $F$, it acts cyclically on $F$. There exists another point $P \in F$, $P\notin E$, which is fixed by $G_F$. The orbit of $P$ by $G$ consists of two points, $P$ and $P'$, such that $P' \in F'$, $P' \notin E$.
We blow-up these two points and blow-down the strict transforms of $F$ and $F'$ to get 
a group of automorphisms of $\mathbb{F}_{n-2}$. We do this $\lfloor n/2\rfloor$ times to obtain $G$ as a group of automorphisms of $\mathbb{F}_0$ or
$\mathbb{F}_1$.

If~$n$ is even, we get in this manner a group of automorphisms of $\mathbb{F}_0=\mathbb{P}^1 \times
\mathbb{P}^1$. 

Note that~$n$ cannot be odd, if the group $\overline{\pi}(G)$ is not cyclic.\ Otherwise, we could conjugate $G$ to a group of automorphisms of $\mathbb{F}_1$ and then to a group of automorphisms of\hspace{0.2 cm}$\Pn$ by blowing-down the exceptional section on a point $Q\in\Pn$. We would get an Abelian subgroup of $\PGLn{3}$ that fixes $Q$, and thus a group with at least three fixed points. In this case, the action on the set of lines passing through $Q$ would be cyclic (see Proposition~\ref{Prp:PGL3Cr}), which contradicts our hypothesis.
\end{proof}

We can now prove the main result of this section:
\begin{prop}\label{Prop:MainConicBundles}
Let $G\subset \Aut(S,\pi)$ be some finite Abelian group of automorphisms of the conic bundle $(S,\pi)$ such that the triple $(G,S,\pi)$ is minimal and no non-trivial element of $G$ fixes a curve of positive genus.
Then, one of the following situations occurs:
\begin{enumerate}
\item[\upshape 1.]
$S$ is a Hirzebruch surface $\mathbb{F}_n$;
\item[\upshape 2.]
$S$ is a del Pezzo surface of degree $5$ or $6$;
\item[\upshape 3.]
The triple $(G,S,\pi)$ is isomorphic to the triple $(\Cs{24},\hat{S_4},\pi)$ of Section~$\ref{Sec:ExampleCs24}$.
\end{enumerate}
If we suppose that the pair $(G,S)$ is minimal, then we are in case $1$ with $n\not=1$ or in case $3$. Moreover, cases $1$ and $2$ are birationally conjugate to automorphisms of $\mathbb{P}^1\times \mathbb{P}^1$ whereas the third is not.
\end{prop}
\begin{proof}
The fact that one of the three cases occurs follows directly from Propositions~\ref{Prp:necDeJi} and~\ref{Prp:HardTwistInv}.

Case $1$ is clearly minimal if and only if $n\not=1$ and Proposition~\ref{Prp:FnAb} shows that it is conjugate to automorphisms of $\mathbb{P}^1\times\mathbb{P}^1$. In the case of del Pezzo surfaces of degree $5$ and $6$, the pair $(G,S)$ is not minimal and the group is respectively birationally conjugate to a subgroup of $\Sym_4\subset\Aut(\Pn)$ (Lemma~\ref{Lem:AutDP5Conic}) or $\Aut(\mathbb{P}^1\times\mathbb{P}^1)$ (Lemma~\ref{Lem:DP6CBPasMin}). If the first situation occurs, since the group is Abelian and not isomorphic to $(\Z{3})^2$ it is diagonalisable and conjugate to a subgroup of $\Aut(\mathbb{P}^1\times\mathbb{P}^1)$ (Proposition~\ref{Prp:PGL3Cr}). Thus, we are done with case $2$.

It remains to show that the pair $(\Cs{24},\hat{S_4})$ is not birationally conjugate to a group of automorphisms of $\mathbb{P}^1\times\mathbb{P}^1$. Let us suppose the contrary, i.e.\ that there exists some $\Cs{24}$-equivariant birational map $\varphi: \hat{S_4}\dasharrow \mathbb{P}^1\times\mathbb{P}^1$ (that conjugates $\Cs{24}$ to a group of automorphisms). Then, $\varphi$ is the composition of $\Cs{24}$-equivariant elementary links (see for example \cite[Theorem 2.5]{bib:Isk5}, or \cite[Theorem 7.7]{bib:Dol}). 
Since our group preserves the conic bundle, the first link is of type $II$, $III$ or $IV$ (in the classical notation of Mori theory). We now study these possibilities and show that it is not possible to go to $\mathbb{P}^1\times\mathbb{P}^1$.

{\it Link of type $II$} - In our case, this link is a birational map of conic bundles, which is the composition of the blow-up of an orbit of $\Cs{24}$, no two points on the same fibre, with the blow-down of the strict transforms of the fibres of the points blown-up. The points must be fixed by the elements of $\Cs{24}$ that act trivially on the basis of the fibration, and thus an orbit has $4$ points, two on $\widetilde{E_1}$ and two on $\widetilde{D_{23}}$. This link conjugates the triple $(\Cs{24},\hat{S_4},\pi)$ to a triple isomorphic to it, by Proposition~\ref{Prp:necDeJi}.

{\it Link of type $III$} - It is the contraction of some set of skew exceptional curves, invariant by $\Cs{24}$. This is impossible since the pair $(\Cs{24},\hat{S_4})$ is minimal (Lemma~\ref{Lem:Cs24Properties}).

{\it Link of type $IV$} - It is a change of the fibration. This is not possible since the surface $\hat{S_4}$ admits only one conic bundle fibration (Corollary~\ref{Cor:S4onefib}).
\end{proof}
\section{Actions on del Pezzo surfaces with fixed part of the Picard group of rank one}
\label{Sec:DelPezzo}
In this section we prove the following result (note that finiteness is not required and that minimality of the action is implied by the condition on $\Pic{S}^G$).
\begin{prop}
\label{Prp:DPrkone}
Let $S$ be a del Pezzo surface, and let $G\subset \Aut(S)$ be an Abelian group such that $\rkPic{S}^G=1$ and no non-trivial element of $G$ fixes a curve of positive genus. Then, one of the following occurs:\begin{enumerate}
\item[\upshape 1.]
$S\cong \Pn$ or $S\cong \mathbb{P}^1\times\mathbb{P}^1$;
\item[\upshape 2.]
$S$ is a del Pezzo surface of degree $5$ and $G\cong \Z{5}$;
\item[\upshape 3.]
$S$ is a del Pezzo surface of degree $6$ and $G\cong \Z{6}$.
\end{enumerate}
Furthermore, in cases $2$ and $3$, the group $G$ is birationally conjugate to a diagonal cyclic subgroup of $\Aut(\Pn)$.
\end{prop}

This will be proved separately for each degree, in Lemmas~\ref{Lem:DP6rk1}, \ref{Lem:DP5rk1}, \ref{Lem:DP4}, \ref{Lem:DP3}, \ref{Lem:DP2} and~\ref{Lem:DP1}. 

\bigskip

 \begin{rema}\label{Rem:Weyl}
A del Pezzo surface $S$ is either $\mathbb{P}^1\times\mathbb{P}^1$ or the blow-up of $0\leq r\leq 8$ points in general position on $\Pn$ (i.e.\ such that no irreducible curve of self-intersection $\leq -2$ appears on $S$). The group $\Pic{S}$ has dimension $r+1$, and its intersection form gives a decomposition $\Pic{S}\otimes \mathbb{Q}=\mathbb{Q}K_S\oplus K_S^{\perp}$; the signature is $(1,-1,...,-1)$.

The group $\Aut(S)$ of automorphisms of a del Pezzo surface $S$ acts on $\Pic{S}$ and preserves the intersection form. This gives an homomorphism of $\Aut(S)\rightarrow \Aut(\Pic{S})$ which is injective if and only if $r>3$, since the kernel is the lift of automorphisms of $\Pn$ that fix the $r$  blown-up points. Furthermore, the image is contained in the Weyl group and is finite (see \cite{bib:DolWeyl}). In particular, the group $\Aut(S)$ is finite if and only if $r>3$. 
 \end{rema}
 
When we have some group action on a del Pezzo surface, we would like to determine the rank of the fixed part of the Picard group. Here are some tools to this end.
\begin{lemm}[Size of the orbits]
\label{Lem:SizeOrbits}
Let $S$ be a del Pezzo surface, which is the blow-up of $1\leq r\leq 8$ points of\hspace{0.2 cm}$\mathbb{P}^2$ in general position,
and let $G \subset \Aut(S)$ be a subgroup of automorphisms with $\rkPic{S}^{G}=1$. Then:
\begin{itemize}
\item
$G\not=\{1\}$;
\item
the size of any orbit of the action of $G$ 
on the set of exceptional divisors is divisible by the degree of $S$, which is $9-r$;
\item
in particular, if the order of $G$ is finite, it is divisible by the degree of $S$.
\end{itemize}
\end{lemm}
\begin{proof}
It is clear that $G \not=\{1\}$, since $\rkPic{S}>1$.
Let $D_1,D_2,...,D_k$ be $k$ exceptional divisors of $S$, forming an orbit of $G$ (the orbit is finite, see Remark~\ref{Rem:Weyl}). 
The divisor $\sum_{i=1}^{k} D_i$ is fixed by $G$ and thus is a multiple of $K_S$. 
We can write $\sum_{i=1}^{k} D_i=a K_S$, for some $a \in \mathbb{Q}$. In fact, since $a K_S$ is effective, we have $a<0$ and $a \in \mathbb{Z}$.
Since the $D_i$'s are irreducible and rational, we deduce from the adjunction formula $D_i (K_S +D_i)=-2$ that $D_i\cdot K_S=-1$.
Hence \begin{center}$K_S \cdot \sum_{i=1}^{k} D_i=\sum_{i=1}^{k}K_S \cdot D_i=-k=K_S \cdot a K_S=a (9-r)$.\end{center} Consequently, the degree $9-r$ divides the size $k$ of the orbit.
\end{proof}
\begin{rema}\label{Rem:DP78}
This lemma shows in particular that $\rkPic{S}^G>1$ if $S$ is the blow-up of $r=1,2$ points of\hspace{0.2 cm}$\Pn$, a result which is obvious when $r=1$, and is clear when $r=2$, since the line joining the two blown-up points is invariant by any automorphism.
\end{rema}

\begin{lemm}\label{Lem:Lefschetz}
Let $S$ be some (smooth projective rational) surface, and let $g \in \Aut(S)$ be some automorphism of finite order. Then, the trace of~$g$ acting on $\Pic{S}$ is equal to $\chi(\Fix(g))-2$, where $\Fix(g)\subset S$ is the set of fixed points of~$g$ and $\chi$ is the Euler characteristic.
\end{lemm}
\begin{proof}
This follows from the topological Lefschetz fixed-point formula, which asserts that the trace of~$g$ acting on $H^{*}(S,\mathbb{Z})$ is equal to $\chi(\Fix(g))$ (this uses the fact that~$g$ is an homeomorphism of finite order). Since $S$ is a complex rational surface, $H^{0}(S,\mathbb{Z})$ and $H^{4}(S,\mathbb{Z})$ have dimension $1$, $H^{2}(S,\mathbb{Z})\cong\Pic{S}$, and $H^{i}(S,\mathbb{Z})=0$ for $i\not=0,2,4$. Since the trace on $H^{2}$ and $H^4$ is $1$, we obtain the result.
\end{proof}
\begin{rema}
This lemma is false if the order of~$g$ is infinite. Take for example the automorphism $(x:y:z)\mapsto (\lambda x:y:z+y)$ of $\Pn$, for any $\lambda\in\K^{*}, \lambda\not=1$. It fixes exactly two points, namely $(1:0:0)$ and $(0:0:1)$, but its trace on $\Pic{\Pn}=\mathbb{Z}$ is $1$.
\end{rema}

We now start the proof of Proposition~\ref{Prp:DPrkone} by studying the cases of del Pezzo surfaces of degree $6$ or $5$.
\begin{lemm}[Actions on the del Pezzo surface of degree $6$]
\label{Lem:DP6rk1}
Let $S_6=\{ \big((x:y:z) , (u:v:w)\big)\ | \ ux=vy=wz\}\subset\Pn\times\Pn$ be the del Pezzo surface of degree $6$ and let $G\subset \Aut(S_6)$ be an Abelian group such that $\rkPic{S_6}^G=1$. Then, $G$ is conjugate in $\Aut(S_6)$ to the cyclic group of order $6$ generated by $\big((x:y:z) , (u:v:w)\big)\mapsto \big((v:w:u) , (y:z:x)\big)$. Furthermore, $G$ is birationally conjugate to a diagonal subgroup of $\Aut(\Pn)$.
\end{lemm}
\begin{proof}
Lemma~\ref{Lem:SizeOrbits} implies that the sizes of the orbits of the action of $G$ on the exceptional divisors are divisible by $6$. The action of $G$ on the hexagon of exceptional divisors is thus transitive, so $G$ contains an element of the form
\begin{center} $g:\big((x:y:z) , (u:v:w)\big) \mapsto \big((\alpha v:\beta w:u) , (\beta y:\alpha z:\alpha\beta x)\big)$,\end{center} where $\alpha,\beta \in \K^{*}$. As the only element of $(\K^{*})^2$ that commutes with~$g$ is the identity (see the description of $\Aut(S_6)=(\K^{*})^2 \rtimes (\Sym_3\times \Z{2})$ in Section~\ref{Sec:DelPezzo6}), $G$ must be cyclic, generated by~$g$. Conjugating it by \begin{center}$\big((x:y:z) , (u:v:w)\big) \mapsto \big((\beta x: y:\alpha z) , (\alpha u:\alpha\beta v:\beta w)\big)$,\end{center}
we may assume that $\alpha=\beta=1$, as stated in the lemma (this shows in particular that $G$ is of finite order). It remains to prove that this automorphism is birationally conjugate to a linear automorphism of the plane.

Denote by $p:S\rightarrow \Pn$ the restriction of the projection on the first factor. This is a birational morphism which is the blow-up of the three diagonal points $A_1,A_2,A_3$ of $\Pn$. 
Consider the birational map $\hat{g}=p gp^{-1}$ of\hspace{0.2 cm}$\Pn$, which is explicitly $\hat{g}:(x:y:z) \dasharrow (xz:xy: yz)$. Since~$g$ is an automorphism of the surface, it fixes the canonical divisor $K_S$, so the birational map
$\hat{g}$ leaves the linear system of cubics of\hspace{0.2 cm}$\Pn$ passing through $A_1,A_2$ and $A_3$ invariant (this can also be verified directly).

Note that $\hat{g}$ fixes exactly one point of\hspace{0.2 cm}$\Pn$, namely $P=(1:1:1)$, and that its action on the projective tangent space $\mathbb{P}(T_{P}(\Pn))$ of\hspace{0.2 cm}$\Pn$ at $P$ is of order $3$, with two fixed points, corresponding to the lines $(x-y)+\omega^k(z-y)=0$, where $\omega=e^{2\ipi/3}$, $k=1,2$. Hence, the birational map $\hat{g}$ preserves the linear system of cubics of $\Pn$ passing through $A_1,A_2$ and $A_3$, which have a double point at $P$ and are tangent to the line $(x-y)+\omega(z-y)=0$ at this point. This linear system thus induces a birational transformation of $\Pn$ that conjugates $\hat{g}$ to a linear automorphism.
\end{proof}

\begin{lemm}[Actions on the del Pezzo surface of degree $5$]
\label{Lem:DP5rk1}
Let $S_5$ be the del Pezzo surface of degree $5$ and let $G\in \Aut(S_5)=\Sym_5$ be an Abelian group such that $\rkPic{S_5}^G=1$. Then, $G$ is cyclic of order $5$. Furthermore, $G$ is birationally conjugate to a diagonal subgroup of $\Aut(\Pn)$.
\end{lemm}
\begin{proof} We use the description of the surface $S_5$ and its automorphisms group $\Aut(S_5)=\Sym_5$ given in Section~\ref{Sec:DelPezzo5}. 
Lemma~\ref{Lem:SizeOrbits} implies that the order of $G$ is divisible by $5$, and thus that $G$ is a cyclic subgroup of $\Sym_5$ of order $5$. Since all such subgroups are conjugate in $\Aut(S_5)=\Sym_5$, we may suppose that $G$ is generated by the lift of the birational transformation $h:(x:y:z) \dasharrow (xy:y(x-z):x(y-z))$ of\hspace{0.2 cm}$\Pn$, that fixes two points of $\Pn$, namely $(\zeta+1:\zeta:1)$, where $\zeta^2-\zeta-1=0$. Denoting one of them by $P$, the linear system of cubics passing through the four blown-up points and having a double point at $P$ is invariant by~$h$. The birational transformation associated to this system thus conjugates~$h$ to a linear automorphism of $\Pn$.\end{proof}
\begin{rema}
The fact that $(x:y:z) \dasharrow (xy:y(x-z):x(y-z))$ is linearisable was proved in \cite{bib:BeB}, using the same argument as above.
\end{rema}
\begin{coro}\label{Coro:K2leq5}
Let $S$ be a rational surface with $(K_S)^2\geq 5$ and let $G\subset \Aut(S)$ be a finite Abelian group. Then $G$ is birationally conjugate to a subgroup of $\Aut(\Pn)$ or $\Aut(\mathbb{P}^1\times\mathbb{P}^1)$.
\end{coro}
\begin{proof}
We may assume that the pair $(G,S)$ is minimal; consequently there are two possibilities (see \cite{bib:Man}, \cite{bib:Isk3} or \cite{bib:Dol}):

{\it $1.$ $S$ is a del Pezzo surface and $\rkPic{S}^G=1$.} Then $S$ is either $\Pn$, $\mathbb{P}^1\times\mathbb{P}^1$ or a del Pezzo surface of degree $6$ or $5$ (Remark~\ref{Rem:DP78}); we apply Lemmas~\ref{Lem:DP6rk1} and~\ref{Lem:DP5rk1} to conclude.

{\it $2.$ $G$ preserves a conic bundle structure on $S$.} Here the number of fibres is at most $3$, hence no element of $G$ fixes a curve of positive genus (Corollary~\ref{Cor:CurveFixDeJ}); we apply Proposition~\ref{Prop:MainConicBundles} to conclude.
\end{proof}
To study del Pezzo surfaces of degree $4$, let us describe their group of automorphisms (note that we do not use the notation $S_d$ for the del Pezzo surfaces of degree $d \leq 4$, because there are many different surfaces of the same degree): 
\begin{lemm}[Automorphism group of del Pezzo surfaces of degree $4$]
\label{Lem:AutDpdeg4}
Let $S$ be a del Pezzo surface of degree $4$ given by the blow-up $\eta:S\rightarrow \Pn$ of five points $A_1,...,A_5\in \Pn$ such that no three are collinear. Setting $E_i=\eta^{-1}(A_i)$ and denoting by $L$ the pull-back by $\eta$ of a general line of $\Pn$, we have:

\begin{enumerate}
\item[\upshape 1.]
There are exactly $10$ conic bundle structures on $S$, whose fibres are respectively $L-E_i$, $-K_S-(L-E_i)$, for $i=1,...,5$.
\item[\upshape 2.]
The action of $\Aut(S)$ on the five pairs of divisors $\{L-E_i,-K_S-(L-E_i)\}$, $i=1,...,5$ gives rise to a split exact sequence 
\begin{center}$
0 \rightarrow \mathbf{F} \rightarrow \Aut(S)\stackrel{\rho}{\rightarrow} \Sym_5,$\end{center}
where $\mathbf{F}=\{(a_1,...,a_5) \in (\mathbb{F}_2)^5\ |\ \sum a_i=0\}\cong(\mathbb{F}_2)^4$, and the automorphism $(a_1,...,a_5)$ permutes the pair $\{L-E_i,-K_S-(L-E_i)\}$ if and only if $a_i=1$.
\item[\upshape 3.]
We have
\begin{center}$
\Aut(S) =\mathbf{F} \rtimes \Aut(S,\eta),$\end{center}
where $\Aut(S,\eta)$ is the lift of the group of automorphisms of $\Pn$ that leave the set $\{A_1,...,A_5\}$ invariant,
 and $\Aut(S,\eta)$ acts on $\mathbf{F}=\{(a_1,...,a_5) \in (\mathbb{F}_2)^5\ |\ \sum a_i=0\}$ by permutation of the $a_i$'s, as it acts on $\{A_1,...,A_5\}$, and as $\rho(\Aut(S))=\rho(\Aut(S,\eta))\subset \Sym_5$ acts on the exceptional pairs.
\item[\upshape 4.]
The elements of $\mathbf{F}$ with two "ones" correspond to quadratic involutions of $\Pn$ and fix exactly $4$ points of $S$. 
\item[\upshape 5.]
The elements of $\mathbf{F}$ with four "ones" correspond to cubic involutions of $\Pn$ and the points of $S$ fixed by these elements form a smooth elliptic curve.
\end{enumerate}
\end{lemm}
\begin{rema}
The group $\mathbf{F}\subset \Aut(S)$ has been studied intensively since 1895 (see 
\cite{bib:SK}, Theorem XXXIII). A modern description of the group as the $2$-torsion of $\PGL(5,\K)$ can be found in \cite[(4.1)]{bib:Bea2}, together with a study of the conjugacy classes of such groups in the Cremona group. For further descriptions of the automorphism groups of these surfaces, see \cite[section~6.4]{bib:Dol} and \cite[section~8.1]{bib:JBTh}.
\end{rema}
\begin{proof}
Let $A=mL-\sum_{i=1}^5 a_i E_i$ be the divisor of the fibre of some conic bundle structure on $S$, for some $m,a_1,...,a_5 \in \mathbb{Z}$. From the relations $A^2=0$ (the fibres are disjoint) and $AK_S=-2$ (adjunction formula) we get:
\begin{equation}\label{eqexcpair}\begin{array}{lll}
\sum_{i=1}^5 {a_i}^2&=&m^2, \vspace{0.1 cm}\\
\sum_{i=1}^5 a_i&=&3m-2.\end{array}
\end{equation} 

As in Lemma~\ref{Lem:10irreduciblecurves}, we have  $(\sum_{i=1}^5 a_i)^2\leq 5\sum_{i=1}^5 {a_i}^2$, which implies here that $(3m-2)^2\leq 5m^2$, that is $4(m^2-3m+1)\leq 0$. As~$m$ is an integer, we must have $1\leq m\leq 2$.
If $m=1$, we replace it in (\ref{eqexcpair}) and see that there exists $i \in \{1,...,5\}$ such that $A=L-E_i$. Otherwise, taking $m=2$ and replacing it in (\ref{eqexcpair}), we see that four of the $a_j$'s are equal to $1$, and one is equal to $0$. This gives the ten conic bundles of assertion~$1$, which are the lift on $S$ of the lines of $\Pn$ passing through one of the $A_i$'s or of the conic passing through four of the $A_i$'s.

The group $\Aut(S)$ acts on the set $\cup_{i=1}^5 \{L-E_i,-K_S-(L-E_i)\}$; since $K_S$ is fixed, this induces an action on the set of five pairs $\{L-E_i,-K_S-(L-E_i)\}$. We denote by $\rho:\Aut(S)\rightarrow \Sym_5$ the corresponding homomorphism. The action of the kernel of $\rho$ on the pairs of conic bundles gives a natural embedding of $\Ker(\rho)$ into $(\mathbb{F}_2)^5$. 

We now prove that $\Ker (\rho)= \{(a_1,...,a_5)\ | \ \sum a_i=0\}=\mathbf{F}$. Acting by a linear automorphism of $\Pn$, we may assume that the points blown-up by $\eta$ are $A_1=(1:0:0)$, $A_2=(0:1:0)$, $A_3=(0:0:1)$, $A_4=(1:1:1)$, $A_5=(a:b:c)$, for some $a,b,c \in \K^{*}$. Then, the birational involution $\tau:(x_0:x_1:x_2) \dasharrow (ax_1x_2:bx_0x_2:cx_0x_1)$ of $\Pn$ lifts as an automorphism $\eta^{-1}\tau\eta\in \Aut(S)$ that acts on $\Pic{S}$ as
\begin{center}
$\left(\begin{array}{cccccc}
0 & -1 & -1& 0 & 0 & -1\\
-1& 0 & -1& 0 & 0 & -1\\
-1 & -1& 0& 0 & 0 & -1\\
0 & 0& 0& 0 & 1 & 0\\
0 & 0& 0& 1& 0 & 0\\
1 & 1& 1& 0 & 0 & 2\end{array}\right)$,\end{center}
with respect to the basis $(E_1,E_2,E_3,E_4,E_5,L)$. It follows from this observation that $\eta^{-1}\tau\eta$ belongs to the kernel of $\rho$, and acts on the pairs of conic bundles as $(0,0,0,1,1) \in (\mathbb{F}_2)^5$. Permuting the roles of the points $A_1,...,A_5$, we get $10$ involutions whose representations in $(\mathbb{F}_2)^5$ have two "ones" and three "zeros". These involutions generate the group $\{(a_1,...,a_5)\ | \ \sum a_i=0\}=\mathbf{F}$. To prove that this group is equal to $\Ker (\rho)$, it suffices to show that $(1,1,1,1,1)$ does not belong to $\Ker (\rho)$. This follows from the fact that $(1,1,1,1,1)$ would send $L=\frac{1}{2}(K_S+\sum_{i=1}^5 (L-E_i))$ on the divisor $\frac{1}{2}(K_S+\sum_{i=1}^5 (-K_S-L+E_i))=\frac{1}{2}(-2L-3 K_S)$, which doesn't belong to $\Pic{S}$. This concludes the proof of assertion~2 (except the fact that the exact sequence is split, which will be proved by assertion~3).

We now prove assertion~3. Let $\sigma \in \Sym_5$ be a permutation of the set $\{1,...,5\}$ in the image of $\rho$ and~$g$ be an automorphism of $S$ such that $\rho(g)=\sigma$. Let $\alpha$ be the element of $\Aut(\Pic{S})$ that sends $E_{i}$ on $E_{\sigma(i)}$ and fixes $L$. Viewing $\Aut(S)$ as a subgroup of $\Aut(\Pic{S})$, the element $g\alpha^{-1} \in \Aut(\Pic{S})$ fixes the five pairs of conic bundles. There exists some element $h \in \mathbf{F}\subset \Aut(S)$ such that $hg\alpha^{-1}$ either fixes the divisor of every conic bundle or permutes the divisors of conic bundles in each pair. The same argument as in the above paragraph shows that this latter possibility cannot occur. Hence $hg\alpha^{-1}$ fixes $L-E_1,...,L-E_5$ and $K_S$. It follows that $hg\alpha^{-1}$ acts trivially on $\Pic{S}$, so $\alpha=hg \in \Aut(S)$, and $\alpha$ is by construction the lift of an automorphism of $\Pn$ that acts on the set $\{A_1,...,A_5\}$ as $\sigma$ does on $\{1,...,5\}$.
Conversely, it is clear that every automorphism $r$ of\hspace{0.2 cm}$\Pn$ which leaves the set $\{A_1,A_2,A_3,A_4,A_5\}$ invariant lifts to the automorphism $\eta^{-1}r\eta$ of $S$ whose action on the pairs of conic bundles is the same as that of $r$ on the set $\{A_1,A_2,A_3,A_4,A_5\}$. This gives assertion~3.

Assertion~4 follows from the above description of some element of $\mathbf{F}\subset \Aut(S)$ with two "ones" as the lift of a birational map of the form $\tau:(x_0:x_1:x_2) \dasharrow (ax_1x_2:bx_0x_2:cx_0x_1)$. As the automorphism $\eta^{-1}\tau\eta\in\Aut(S)$ does not leave any exceptional divisor invariant, its fixed points are the same as those of $\tau$, which are the four points $(\alpha:\beta:\gamma)$, where $\alpha^2=a,\beta^2=b,\gamma^2=c$. 

It remains to prove the last assertion. Note that the element $h=(0,1,1,1,1)\in \Aut(S)$ fixes the divisor $L-E_1$, hence acts on the associated conic bundle structure. Furthermore, the four singular fibres of this conic bundle, $\{L-E_1-E_i,E_i\}$, for $i=2,...,5$, are invariant by~$h$ and this element switches the two components of each fibre. This shows that the action of~$h$ on the basis of the fibration is trivial, so the restriction of~$h$ on each fibre is an involution of $\mathbb{P}^1$ which fixes two points. On each singular fibre, exactly one point is fixed, which is the singular point of the fibre. The situation is similar for the other elements with four "ones" (in fact, the involutions described here are \emph{twisting involutions}, see Lemma~\ref{Lem:DeJI}). 
\end{proof}

\begin{lemm}[Actions on the del Pezzo surfaces of degree $4$]
\label{Lem:DP4}
Let $S$ be a del Pezzo surface of degree $4$, and let $G\in \Aut(S)$ be an Abelian group such that $\rkPic{S}^G=1$. Then, $G$ contains an involution that fixes an elliptic curve.
\end{lemm}
\begin{proof}
We keep the notation of Lemma~\ref{Lem:AutDpdeg4} for $\eta:S\rightarrow \Pn, \Aut(S,\eta), \rho,\mathbf{F},...$ and denote by $H$ the group $G \cap \mathbf{F}=G\cap \Ker \rho$. 
We will prove that $H$ contains an element of $\mathbf{F}$ with four "ones", which is an involution that fixes an elliptic curve (Lemma~\ref{Lem:AutDpdeg4}).

 The group $\rho(G)\subset \rho(\Aut(S))\cong \Aut(S,\eta)$ is isomorphic to a subgroup of $\Aut(S,\eta)$. The group $\Aut(S,\eta)$ is the lift of the group of automorphisms of $\Pn$ that leave the set $\{A_1,...,A_5\}$ invariant (Lemma~\ref{Lem:AutDpdeg4}). The restriction of this group to the conic of $\Pn$ passing through the five points is a subgroup of $\PGL(2,\K)$ that leaves five points invariant. Since $\rho(G)$ is finite and Abelian, it is cyclic, of order at most $5$. We consider the different possibilities.
 
 {\it The order of $\rho(G)$ is $1$}. This implies that $G \subset \mathbf{F}$. If $G$ contains an element with four "ones", we are done. Otherwise, up to conjugation $G$ is a subset of the group generated by $(1,1,0,0,0)$ and $(1,0,1,0,0)$, and fixes $L-E_4$ and $L-E_5$ (thus $\rkPic{S}^G>1$).
 
 {\it The order of $\rho(G)$ is $2$}. Up to a change of numbering, $\rho(G)$ is generated by $(1\ 2)(3\ 4)$; since $G$ is Abelian, we find that $H\subset V=\{(a,a,b,b,0) \ | \ a,b\in \mathbb{F}_2\}$. Let $g=((a,b,c,d,e),(1\ 2)(3\ 4)) \in G$ be such that $\rho(g)=(1\ 2)(3\ 4)$. We may suppose that $e=1$ (otherwise, the group $G$ would fix $L-E_5$ and we would have $\rkPic{S}^G\geq 2$.) Conjugating by $((0,b,0,d,b+c),id)$ we may assume that $g=((a+b,0,c+d,0,1),(1\ 2)(3\ 4))$. In fact, since $a+b+c+d+e=0$, we have $g=((\alpha,0,1+\alpha,0,1),(1\ 2)(3\ 4))$, where $\alpha=a+b=c+d+1 \in \mathbb{F}_2$. 
If $\alpha=1$, then~$g$ has order $4$ and fixes the divisor $2L-E_3-E_4$, thus $G$ cannot be equal to $<g>$ and it follows that $V\subset G$; in particular the element $(1,1,1,1,0)$ is contained in $G$. 
If $\alpha=0$, then $<g>$ fixes $2L-E_1-E_2$, so once again $G$ contains $V$.

 {\it The order of $\rho(G)$ is $3$}. In this case, $\rho(G)$ is generated by a $3$-cycle, namely $(1\ 2 \ 3)$; then $H$ must be a subgroup of $V= \{ (a,a,a,b,a+b) \ | \ a,b \in \mathbb{F}_2\}$. The order of $G$ must be a multiple of $4$, by Lemma~\ref{Lem:SizeOrbits}, hence $H=V$, and thus $G$ contains the element $(1,1,1,1,0)$.

 {\it The order of $\rho(G)$ is $4$}. Then $\rho(G)$ is generated by $(1\ 2 \ 3\ 4)$, so $H$ must be a subgroup of $V=<(1,1,1,1,0)>$. Let $g=((a,b,c,d,e),(1\ 2\ 3\ 4)) \in G$ be such that $\rho(g)=(1\ 3\ 2\ 4)$. Conjugating the group by $((a,a+b,a+b+c,0,a+c),id)$, we may suppose that $g=((0,0,0,e,e),(1\ 3\ 2\ 4))$. 
If $e=1$, then $g^4=(1,1,1,1,0)\in G$.
If $e=0$, the element~$g$ belongs to $H_{S}$, so it fixes the divisors $L$ and $E_5$. As the group $V$ fixes $L-E_5$, the rank of $\Pic{S}^{G}$ cannot be $1$.

 {\it The order of $\rho(G)$ is $5$}. Then, $\rho(G)$ is generated by a $5$-cycle and $H=\{1\}$. The rank of $\Pic{S}^{H}$ cannot be $1$, by Lemma~\ref{Lem:SizeOrbits}.
\end{proof}

Before studying the case of del Pezzo surfaces of degree $\leq 3$, we remind the reader of some classical embeddings of these surfaces.

\begin{rema}\label{Rem:EmbeddingsDP}
Recall (\cite{bib:Kol}, Theorem III.3.5) that a del Pezzo surface of degree $3$ (respectively $2$, $1$) is isomorphic to a smooth hypersurface of degree $3$ (respectively $4$, $6$) in the projective space $\mathbb{P}^3$ (respectively in $\mathbb{P}(1,1,1,2)$, $\mathbb{P}(1,1,2,3)$). Furthermore, in each of the $3$ cases, any automorphism of the surface is the restriction of an automorphism of the ambient space.
We will use these classical embeddings, take $w,x,y,z$ as the variables on the projective spaces, and denote by $\DiaG{\alpha}{\beta}{\gamma}{\delta}$ the automorphism $(w:x:y:z) \mapsto (\alpha w:\beta x:\gamma y:\delta z)$. Note that a del Pezzo surface of degree $4$ is isomorphic to the intersection of two quadrics in~$\mathbb{P}^4$, but we will not use this here. 
  \end{rema}

\begin{lemm}[Actions on the del Pezzo surfaces of degree $3$]
\label{Lem:DP3}
Let $S$ be a del Pezzo surface of degree $3$, and let $G\in \Aut(S)$ be an Abelian group such that $\rkPic{S}^G=1$. Then, $G$ contains an element of order $2$ or $3$ that fixes an elliptic curve of $S$.
\end{lemm}
\begin{proof}
Lemma~\ref{Lem:SizeOrbits} implies that the order of $G$ is divisible by $3$, so $G$ contains an element of order $3$. 
We view $S$ as a cubic surface in $\mathbb{P}^3$, and $\Aut(S)$ as a subgroup of $\PGLn{4}$ (see Remark~\ref{Rem:EmbeddingsDP}).
There are three kinds of elements of order $3$ in $\PGLn{4}$, depending on the nature of their eigenvalues. Setting $\omega=e^{2\ipi/3}$, there are elements with one eigenvalue of multiplicity $3$ (conjugate to $\DiaG{1}{1}{1}{\omega}$, or its inverse), elements with two eigenvalues of multiplicity $2$ (conjugate to $\DiaG{1}{1}{\omega}{\omega}$) and elements with three distinct eigenvalues (conjugate to $\DiaG{1}{1}{\omega}{\omega^2}$). We consider the three possibilities.

\textit{Case a: $G$ contains an element of order $3$ with one eigenvalue of multiplicity~$3$.}\\
The element $\DiaG{1}{1}{1}{\omega}$ fixes the hyperplane $z=0$, whose intersection with the surface $S$ is an elliptic curve (because $\Fix(g)\subset S$ is smooth). Thus, we are done. 

\textit{Case b: $G$ contains an element of order $3$ with two eigenvalues of multiplicity~$2$.}\\
With a suitable choice of coordinates, we may assume that this element is \begin{center}$g=\DiaG{1}{1}{\omega}{\omega}$.\end{center} Since $S$ is smooth, its equation $F$ is of degree at least $2$ in each variable, which implies that $F(w,x,\omega y,\omega z)=F(w,x,y,z)$ (the eigenvalue is $1$); up to a change of coordinates $F=w^3+x^3+y^3+z^3$, which means that $S$ is the Fermat cubic surface. 
The group of automorphisms of $S$ is $(\Z{3})^3 \rtimes Sym_4$ and the centraliser of~$g$ in it is $(\Z{3})^3 \rtimes V$, where $V\cong (\Z{2})^2$ is the subgroup of $Sym_4$ generated by the two transpositions $(w,\ x)$ and $(y,\ z)$. The structure of the centraliser gives rise to an exact sequence
\begin{center}$\begin{array}{ccccccc}
1\rightarrow &(\Z{3})^3 &\rightarrow& (\Z{3})^3 \rtimes V& \stackrel{\gamma}{\rightarrow}& V &\rightarrow 1 \\
& \cup & & \cup & & \cup \\
1\rightarrow & G \cap (\Z{3})^3& \rightarrow& G& \rightarrow &\gamma(G)& \rightarrow 1.
\end{array}$\end{center}
We may suppose that $G$ contains no element of order $3$ with an eigenvalue of multiplicity $3$, since this case has been studied above (case a). There are then three possibilities for $G \cap (\Z{3})^3$, namely $<g>$, $<g,\DiaG{1}{\omega}{1}{\omega}>$ and $<g,\DiaG{1}{\omega}{\omega}{1}>$. The last is conjugate to the second by the automorphism $(y,\ z)$.
Note that~$g$ preserves exactly $9$ of the $27$ lines on the surface; these are $\{w+\omega^i x= y+ \omega^j z=0\}$, for $0\leq i,j \leq 2$. If $G \cap (\Z{3})^3$ is equal to $<g>$, then $G/<g>\cong \gamma(G)$ has order $1,2$ or $4$ and thus $G$ leaves at least one of the $9$ lines invariant, whence $\rkPic{S}^G>1$.
If $G \cap (\Z{3})^3$ is the group $H=<g,\DiaG{1}{\omega}{1}{\omega}>$ we have $G=H$, since the centraliser of $H$ in $(\Z{3})^3 \rtimes V$ is the group $(\Z{3})^3$. As the set of three skew lines $\{w+\omega^i x=y+\omega^i z=0\}$ for $0\leq i \leq 2$ is an orbit of $H$, the rank of $\Pic{S}^{G}$ is strictly larger than $1$.

\textit{Case c: $G$ contains an element~$g$ of order $3$ with three distinct eigenvalues.}\\
We may suppose that $g=\DiaG{1}{1}{\omega}{\omega^2}$. Note that the action of~$g$ on $\mathbb{P}^3$ fixes the line $L_{yz}$ of equation $y=z=0$ and thus the whole group $G$ leaves this line invariant. If $L_{yz}\subset S$, the rank of $\rkPic{S}^G$ is at least $2$. Otherwise, the equation of $S$ is of the form $L_3(w,x)+L_1(w,x)yz+y^3+z^3=0$, where $L_3$ and $L_1$ are homogeneous forms of degree respectively $3$ and $1$, and $L_3$ has three distinct roots, so $\Fix(g)=S\cap L_{yz}$. Since~$g$ fixes exactly three points, the trace of its action on $\Pic{S}\cong \mathbb{Z}^7$ is $1$ (Lemma~\ref{Lem:Lefschetz}) and thus $\rkPic{S}^g>1$, which implies that $G\not= <g>$.

Note that every subgroup of $\PGLn{4}$ isomorphic to $(\Z{3})^2$ contains an element with only two distinct eigenvalues, so we may assume that $G$ contains only two elements of order $3$, which are~$g$ and $g^2$. This implies that the action of $G$ on the three points of $L_{yz}\cap S$ gives an exact sequence
\begin{center}$1\rightarrow <g> \rightarrow G\rightarrow \Sym_3,$\end{center}
where the image on the right is a transposition. The group $G$ thus contains an element of order $2$, that we may assume to be diagonal of the form $(w:x:y:z)\mapsto (-w:x:y:z)$ and that fixes the elliptic curve which is the trace on $S$ of the plane $w=0$.
\end{proof}

\begin{lemm}[Actions on the del Pezzo surfaces of degree $2$]
\label{Lem:DP2}
Let $S$ be a del Pezzo surface of degree $2$, and let $G\in \Aut(S)$ be an Abelian group such that $\rkPic{S}^G=1$. Then, $G$ contains either the Geiser involution (that fixes a curve isomorphic to a smooth quartic curve) or an element of order $2$ or $3$ that fixes an elliptic curve.
\end{lemm}
\begin{proof}
We view $S$ as a surface of degree four in the weighted projective space $\mathbb{P}(2,1,1,1)$ (see Remark~\ref{Rem:EmbeddingsDP}). Note that the projection on the last three coordinates gives $S$ as a double covering of $\mathbb{P}^2$ ramified over a smooth quartic curve~$Q$. 

Lemma~\ref{Lem:SizeOrbits} implies that the order of $G$ is divisible by $2$, so $G$ contains an element~$g$ of order $2$. 

If the element~$g$ is the involution induced by the double covering (classically called the \emph{Geiser involution}), we are done; otherwise we may assume that~$g$ acts on $\mathbb{P}(2,1,1,1)$ as $g:(w:x:y:z) \mapsto (\epsilon w:x:y:-z)$, where $\epsilon=\pm 1$, and the equation of $S$ is $w^2=z^4+L_2(x,y)z^2+L_4(x,y)$, where $L_i$ is a form of degree $i$, and $L_4$ has four distinct roots. The trace on $S$ of the equation $z=0$ defines an elliptic curve $L_z\subset S$. If $\epsilon=1$, then~$g$ fixes the curve $L_z$ and we are done; we therefore assume that $\epsilon=-1$.

 If $G$ contains another involution, we diagonalise the group generated by these two involutions and see that one element of the group fixes either an elliptic curve or the smooth quartic curve, so we may assume that~$g$ is the only involution of $G$.

Note that~$g$ fixes exactly four points of $S$, which are the points of intersection of $L_z$ with the quartic $Q$ (of equation $w=0$). The trace of~$g$ on $\Pic{S}\cong\mathbb{Z}^8$ is thus equal to $2$ (Lemma~\ref{Lem:Lefschetz}), whence $\rkPic{S}^g=5$ and $G\not=<g>$.

The group $G$ acts on the line $z=0$ of $\Pn$ and on the four points of $L_z\cap Q$. Since~$g$ is the only element of order $2$ of $G$, the action of $G$ on these four aligned points has order $3$ and thus, we may assume that $L_4(x,y)=x(x^3+\lambda y^3)$ and that there exists an element~$h$ of $G$ that acts as $(w:x:y:z) \mapsto (\alpha w:x:e^{2\ipi/3}y:\beta z)$, with $\alpha^2=\beta^4=1$. We find that $h^4$ is an element of order $3$ that fixes the elliptic curve which is the trace on $S$ of the equation $y=0$.
\end{proof}

\begin{lemm}[Actions on the del Pezzo surfaces of degree $1$]
\label{Lem:DP1}
Let $S$ be a del Pezzo surface of degree $1$, and let $G\in \Aut(S)$ be an Abelian group such that $\rkPic{S}^G=1$. Then, some non-trivial element of $G$ fixes a curve of $S$ of positive genus.
\end{lemm}
\begin{proof}
We view $S$ as a surface of degree six in the weighted projective space $\mathbb{P}(3,1,1,2)$ (see Remark~\ref{Rem:EmbeddingsDP}). Up to a change of coordinates, we may assume that the equation is 
\begin{center}$w^2=z^3+zL_4(x,y)+L_6(x,y),$\end{center}
where $L_4$ and $L_6$ are homogeneous forms of degree $4$ and $6$ respectively.
The embedding of $S$ into $\mathbb{P}(3,1,1,2)$ is given by $|-3K_S| \times |-K_S| \times |-2K_S|$, which implies that $G$ is a subgroup of $P(\GL(1,\K)\times \GL(2,\K)\times \GL(1,\K))$. The projection $(w:x:y:z)\dasharrow (x:y)$ is an elliptic fibration generated by $|-K_S|$, and has one base-point, namely $(1:0:0:1)$, which is fixed by $\Aut(S)$. This projection induces an homomorphism $\rho:\Aut(S)\rightarrow \Aut(\mathbb{P}^1)=\PGL(2,\K)$. Note that the kernel of $\rho$ is generated by the Bertini involution $w\mapsto -w$ (and the element $z\mapsto \omega z$ ($\omega=e^{2\ipi/3}$) if $L_4=0$) and is hence cyclic of order $2$ (or $6$). Furthermore, any element of this kernel fixes a curve of positive genus.

We assume that no non-trivial element of $G$ fixes a curve of positive genus. This implies that $G$ is isomorphic to $\rho(G)\subset\Aut(\mathbb{P}^1)$, and thus is either cyclic or isomorphic to $(\Z{2})^2$. Since the lift of this latter group in $\Aut(S)$ is not Abelian, $G$ is cyclic. We use the Lefschetz fixed-point formula (Lemma~\ref{Lem:Lefschetz}) to deduce the eigenvalues of the action of elements of $G$ on $\Pic{S}\cong \mathbb{Z}^9$. For any element $g\in G$, $g\not=1$, $\Fix(g)$ contains the point $(1:0:0:1)$ and is the disjoint union of points and lines. Thus $\chi(\Fix(g))\geq 1$ and so the trace of~$g$ on $\Pic{S}$ is at least $-1$ (Lemma~\ref{Lem:Lefschetz}).

{\it Elements of order $2$}: The eigenvalues are $<1^a,(-1)^b>$ with $a\geq 4$, $b\leq 5$.

{\it Elements of order $3$}: The eigenvalues are $<1^a,(\omega)^b,(\omega^2)^b>$ with $a\geq 3$, $b\leq 3$.

{\it Elements of order $4$}: The eigenvalues are $<1^a,(-1)^b,(\im)^c,(-\im)^c>$ with $a\geq b-1$. Furthermore, the information on the square induces that $a+b\geq 4$, so $a\geq 2$.

{\it Elements of order $5$}: The eigenvalues are $<1^5,l_1,l_2,l_3,l_4>$, where $l_1,...,l_4$ are the four primitive $5$-th roots of unity.

{\it Elements of order $6$}: The eigenvalues are $<1^a,(-1)^b,(\omega)^c,(\omega^2)^c$, $(-\omega)^d$, $(-\omega^2)^d>$, where $a-b-c+d\geq -1$. Computing the square and the third power, we find respectively $a+b\geq 3, c+d\leq 3$ and $a+2c\geq 4$, $b+2d\leq 5$. This implies that $a\geq 2$. Indeed, if $a=1$, we get $b,c\geq 2$ and thus $d\leq 1$, which contradicts the fact that the trace $a-b-c+d$ is at least $-1$.

Since $\rkPic{S}^G=1$, the order of the cyclic group $G$ is at least $7$. As the action of $G$ leaves $L_4$ and $L_6$ invariant, both $L_6$ and $L_4$ are monomials. If some double root of $L_6$ is a root of $L_4$, the surface is singular, so up to an exchange of coordinates we may suppose that $L_4=x^4$ and either $L_6=xy^5$ or $L_6=y^6$.

In the first case, the equation of the surface is $w^2=z^3+x^4z+xy^5$ whose group of automorphisms $\Aut(S)$ is isomorphic to $\Z{20}$, generated by $\DiaG{\im}{1}{\zeta_{10}}{-1}$, and contains the Bertini involution. No subgroup of $\Aut(S)$ fullfills our hypotheses.

In the second case, the equation of the surface is $w^2=z^3+x^4z+y^6$, whose group of automorphisms is isomorphic to $\Z{2}\times\Z{12}$, generated by the Bertini involution and $g=\DiaG{\im}{1}{\zeta_{12}}{-1}$. The only possibility for $G$ is to be equal to $<g>$. Since $g^4=\DiaG{1}{1}{\omega}{1}$ fixes an elliptic curve, we are done.
\end{proof}

Proposition~\ref{Prp:DPrkone} now follows, using all the lemmas proved above.
\section{The results}
We now prove the five theorems stated in the introduction.
\begin{proof}[Proof of Theorem~$\ref{Thm:Classifmin}$]
Since the pair $(G,S)$ is minimal, either $\rkPic{S}^G=1$ and $S$ is a del Pezzo surface, or $G$ preserves a conic bundle structure (see \cite{bib:Man}, \cite{bib:Isk3} or \cite{bib:Dol}). 

In the first case, either $S\cong \Pn$, or $S\cong \mathbb{P}^1\times\mathbb{P}^1$ or $S$ is a del Pezzo surface of degree $d=5$ or $6$ and $G\cong \Z{d}$ (Proposition~\ref{Prp:DPrkone}).

In the second case, either $S$ is a Hirzebruch surface or the pair $(G,S)$ is the pair $(\Cs{24},\hat{S_4})$ of Section~\ref{Sec:ExampleCs24} (Proposition~\ref{Prop:MainConicBundles}).
\end{proof}

\begin{proof}[Proof of Theorem~$\ref{Thm:NonCyclic}$]
No non-trivial element of $\Aut(\Pn), \Aut(\mathbb{P}^1\times\mathbb{P}^1)$ or $\Cs{24}$ fixes a non-rational curve (the first two cases are clear, the last one follows from Lemma~\ref{Lem:Cs24Properties}).

Conversely, suppose that $G$ is a finite Abelian subgroup of the Cremona group such that no non-trivial element fixes a curve of positive genus. Since $G$ is finite, it is birationally conjugate to a group of automorphisms of a rational surface $S$ (see for example \cite[Theorem 1.4]{bib:DFE} or \cite{bib:Dol}). Then, we assume that the pair $(G,S)$ is minimal and use the classification of Theorem~\ref{Thm:Classifmin}. 

If $S$ is an Hirzebruch surface, the group is birationally conjugate to a subgroup of $\Aut(\mathbb{P}^1\times\mathbb{P}^1)$ (Proposition~\ref{Prp:FnAb}). If $S$ is a del Pezzo surface, the group $G$ is birationally conjugate to a subgroup of $\Aut(\mathbb{P}^1\times\mathbb{P}^1)$ or $\Aut(\Pn)$, by Proposition~\ref{Prp:DPrkone}.
Otherwise, the pair $(G,S)$ is isomorphic to the pair $(\Cs{24},\hat{S_4})$.

It remains to show that the group $\Cs{24}$ is not birationally conjugate to a subgroup of $\Aut(\mathbb{P}^1\times\mathbb{P}^1)$ or $\Aut(\Pn)$. Since the group is isomorphic to $\Z{2}\times\Z{4}$, only the case of $\Aut(\mathbb{P}^1\times\mathbb{P}^1)$ need be considered (see Section~\ref{Sec:P1P2C2}). This was proved in Proposition~\ref{Prop:MainConicBundles}.
\end{proof}
\begin{proof}[Proof of Theorem~$\ref{Prp:TheClassification}$]
By Theorem~\ref{Thm:NonCyclic}, $G$ is birationally conjugate either to a subgroup of $\Aut(\Pn)$, or of $\Aut(\mathbb{P}^1\times\mathbb{P}^1)$, or to the group $\Cs{24}$.

The group $\Cs{24}$ is case $[8]$. The finite Abelian subgroups of $\Aut(\Pn)$ are conjugate to the groups of case $[1]$ or $[9]$ (Proposition~\ref{Prp:PGL3Cr}). The finite Abelian subgroups of $\Aut(\mathbb{P}^1\times\mathbb{P}^1)$ are conjugate to the groups of cases $[1]$ through $[7]$ (Proposition~\ref{Prp:AutP1P1B}).

 It was proved in Proposition~\ref{Prp:AutP1P1B} that cases $[1]$ through $[7]$ are distinct. In Proposition~\ref{Prop:MainConicBundles} we showed that $[8]$ ($\Cs{24}$) is not birationally conjugate to any groups of cases $[1]$ through $[7]$. Finally, the group $[9]$ is isomorphic only to $[1]$, but is not birationally conjugate to it (Proposition~\ref{Prp:PGL3Cr}). This completes the proof that the distincts cases given above are not birationally conjugate.
\end{proof}
The proof of Theorem~\ref{Thm:Cyclic} follows directly from Theorem~\ref{Prp:TheClassification}, and Theorem~\ref{Thm:Roots} is a corollary of Theorem~\ref{Thm:Cyclic}.
\section{Other kinds of groups}
\label{Sec:Counter}
Our main interest up to now was in finite Abelian subgroups of the Cremona group. In this section, we give some examples in the other cases, in order to show why the hypothesis "finite", respectively "Abelian", is necessary to ensure that condition $(F)$ (no curve of positive genus is fixed by a non-trivial element) implies condition $(M)$ (the group is birationally conjugate to a group of automorphisms of a minimal surface). We refer to the introduction for more details.

Finiteness is important since it imposes that the group is conjugate to a group of automorphisms of a projective rational surface. This is not the case if the group is not finite (see for example \cite{bib:JBTh}, Proposition 2.2.4). 
\begin{lemm}\label{Lem:GeneralNonLinearisable}
Let $\varphi:\Pn\dasharrow\Pn$ be a quadratic birational transformation with three proper base-points, and such that $\deg(\varphi^n)=2^n$ for each integer $n\geq 1$. Then, the following occur:

\begin{enumerate}
\item[\upshape 1.]
no pencil of curves is invariant by~$\varphi$;
\item[\upshape 2.]
 $\varphi$ is not birationally conjugate to an automorphism of $\Pn$ or of  $\mathbb{P}^1\times\mathbb{P}^1$.
\end{enumerate}
\end{lemm}
\begin{proof}Denote by $A_1,A_2,A_3$ the three base-points of $\varphi$ and by $B_1,B_2,B_3$ those of $\varphi^{-1}$. Up to a change of coordinates, we may suppose that $A_1=(1:0:0)$, $A_2=(0:1:0)$ and $A_3=(0:0:1)$. The birational transformation $\varphi$ is thus the composition of the standard quadratic transformation $\sigma: (x:y:z)\dasharrow (yz:xz:xy)$ with a linear automorphism $\tau\in\Aut(\Pn)$ that sends $A_i$ on $B_i$ for $i=1,2,3$. 

Let $\Lambda$ be some pencil of curves, and assume that $\varphi(\Lambda)=\Lambda$. We will prove that some base-point of $\Lambda$ is sent by~$\varphi$ on an orbit of infinite order. The condition $\deg(\varphi^n )=2^n$ is equivalent to saying that for $i=1,2,3,$ the sequence $B_i, \varphi(B_i),...,\varphi^n(B_i),...$ is well-defined, i.e.\ that $\varphi^m(B_i)$ is not equal to $A_j$ for any $i,j\in\{1,2,3\},m\in \mathbb{N}$. Denote by $\alpha_1,\alpha_2,\alpha_3,\beta_1,\beta_2,\beta_3$ the multiplicity of $\Lambda$ at respectively $A_1$, $A_2$, $A_3$, $B_1$, $B_2$, $B_3$ and by~$n$ the degree of the curves of $\Lambda$. The curves of the pencil $\varphi(\Lambda)$ thus have degree $2n-\alpha_1-\alpha_2-\alpha_3$. Since $\Lambda$ is invariant, $n=\alpha_1+\alpha_2+\alpha_3$, so at least one of the $\alpha_i$'s is not equal to zero.
The equality $n=\alpha_1+\alpha_2+\alpha_3$ implies that the curves of $\sigma(\Lambda)$ have multiplicity $\alpha_i$ at $A_i$, so the curves of $\varphi(\Lambda)$ have multiplicity $\alpha_i$ at $B_i$, whence $\alpha_i=\beta_i$ for $i=1,2,3$. Since $\Lambda$ passes through $B_i$ with multiplicity $\alpha_i$, the pencil $\varphi(\Lambda)=\Lambda$ passes through $\varphi(B_i)$ with multiplicity $\alpha_i$ for $i=1,2,3$. Continuing in this way, we see that $\Lambda$ passes through $\varphi^n(B_i)$ with multiplicity $\alpha_i$ for each $n\in\mathbb{N}$. Consequently, $\Lambda$ has infinitely many base-points, which is not possible. This establishes the first assertion.

The second assertion follows directly, as each automorphism of $\Pn$ or $\mathbb{P}^1\times\mathbb{P}^1$ leaves a pencil of rational curves invariant. 
\end{proof}
\begin{coro}
The group generated by a very general quadratic transformation is a infinite cyclic group satisfying $(F)$ but not $(M)$.
\end{coro}
\begin{proof}
The condition $\deg(\varphi^n)=2^n$, $n\in \mathbb{N}$ is satisfied for all quadratic transformations, except for a countable set of proper subvarieties. Consequently condition $(F)$ is not satisfied (Lemma \ref{Lem:GeneralNonLinearisable}) for a very general quadratic transformation.

Let~$n$ be some positive integer and write $\varphi^n:(x:y:z)\dasharrow (f_1(x,y,z):f_2(x,y,z):f_3(x,y,z))$, for some homogeneous polynomials $f_i$ of degree $2^n$. The set of points fixed by $\varphi^n$ belongs to the intersection of the curves with equations $xf_2-yf_1$, $xf_3-zf_1$ and $yf_3-zf_2$. In general, there is only a finite number of points; this yields condition $(F)$.
\end{proof}
In fact, the argument of Lemma~\ref{Lem:GeneralNonLinearisable} works for any very general birational transformation of $\Pn$, since this is a composition of quadratic transformations. We thus find infinitely many cyclic subgroups of the Cremona group that are not birationally conjugate to a group of automorphisms of a minimal surface although none of their non-trivial elements fixes a non-rational curve. The implication $(F)\Rightarrow (M)$ is therefore false for general cyclic groups.

\bigskip

We now study the finite non-Abelian subgroups and provide, in this case, many examples satisfying $(F)$ but not $(M)$:
\begin{lemm}Let $S_6=\{ \big((x:y:z) , (u:v:w)\big)\ | \ ux=vy=wz\}\subset\Pn\times\Pn$ be the del Pezzo surface of degree $6$. Let $G\cong \Sym_3\times\Z{2}$ be the subgroup of automorphisms of $S_6$ generated by \begin{center}$\begin{array}{rcl}\big((x:y:z) , (u:v:w)\big)& \mapsto& \big((u:v:w), (x:y:z)\big),\\
\big((x:y:z) , (u:v:w)\big)& \mapsto& \big((y:x:z), (v:u:w)\big),\\
\big((x:y:z) , (u:v:w)\big)& \mapsto& \big((z:y:x), (w:v:u)\big).\end{array}$\end{center}
Then no non-trivial element of $G$ fixes a curve of positive genus, and $G$ is not birationally conjugate to a group of automorphisms of a minimal surface.
\end{lemm}
\begin{proof}Since every non-trivial element of finite order of $\Aut(S_6)$ is birationally conjugate to a linear automorphism of $\Pn$ (Corollary~\ref{Coro:K2leq5}), no such element fixes a curve of positive genus. The description of every $G$-equivariant elementary link starting from $S_6$ was given by Iskovskikh in \cite{bib:IskS3Z2}. This shows that this group is not birationally conjugate to a group of automorphisms of a minimal surface.\end{proof}
\begin{lemm}Let $S_5$ be the del Pezzo surface of degree $5$. Let $G\cong \Sym_5$ be the whole group $\Aut(S_5)$.
Then no non-trivial element of $G$ fixes a curve of positive genus, and $G$ is not birationally conjugate to a group of automorphisms of a minimal surface.
\end{lemm}
\begin{proof}Since every non-trivial element of $\Aut(S_5)$ is birationally conjugate to a linear automorphism of $\Pn$ (Corollary~\ref{Coro:K2leq5}), such an element does not fix a curve of positive genus. Suppose that there exists some $G$-equivariant birational transformation $\varphi:S_5\dasharrow \tilde{S}$ where $\tilde{S}$ is equal to $\Pn$ or $\mathbb{P}^1\times\mathbb{P}^1$. We decompose $\varphi$ into $G$-equivariant elementary links (see for example \cite{bib:Isk5}, Theorem 2.5). The classification of elementary links (\cite{bib:Isk5}, Theorem 2.6) shows that a link $S_5\dasharrow S'$ is 
either a Bertini or a Geiser involution (and in this case $S'=S_5$, and thus this link conjugates $G$ to itself), or the composition of the blow-up of one or two points, and the contraction of $5$ curves to respectively $\mathbb{P}^1\times\mathbb{P}^1$ or $\Pn$. It remains to show that no orbit of $G$ has size $2$ or $1$, to conclude that these links are not possible. This follows from the fact that the actions of $\Sym_5,\Alt_5 \subset G$ on $S_5$ are fixed-point free (Proposition~\ref{Prp:AutS5}).
\end{proof}

Finally, the way to find more counterexamples is to look at groups acting on conic bundles. The generalisation of the example $\Cs{24}$ gives many examples of non-Abelian finite groups. Here is the simplest family:
\begin{lemm}
Let~$n$ be some positive integer, and let $G$ be the group of birational transformations of $\Pn$ generated by 

\begin{center}$\begin{array}{rcl}g_1&:(x:y:z)\dasharrow& (yz:xy:-xz),\\
g_2&:(x:y:z)\dasharrow& (yz(y-z):xz(y+z):xy(y+z)),\\
h&:(x:y:z)\dasharrow& (e^{2\ipi/2n}x:y:z).\end{array}$\end{center} 
Then, $G$ preserves the pencil $\Lambda$ of lines passing through $(1:0:0)$ and the corresponding action gives rise to a non-split exact sequence
\begin{center}$1\rightarrow <h>\cong \Z{2n}\rightarrow G\rightarrow (\Z{2})^2\rightarrow 1.$\end{center}
In particular, the group $G$ has order $8n$.
Furthermore, no non-trivial element of $G$ fixes a curve of positive genus, and $G$ is not birationally conjugate to a group of automorphisms of a minimal surface.
\end{lemm}
\begin{proof}
Firstly, since $g_1$ and $g_2$ generate the group $\Cs{24}$, which is not birationally conjugate to a group of automorphisms of a minimal surface, this is also the case for $G$.

Secondly, we compute that $(g_1)^2=(g_2)^2=(h)^n$ is the birational transformation $(x:y:z)\mapsto (-x:y:z)$. The maps $g_1$ and $g_2$ thus have order $4$ and~$h$ has order $2n$.

Thirdly, every generator of $G$ preserves the pencil $\Lambda$ of lines passing through $(1:0:0)$. The action of $g_1,g_2$ and~$h$ on this pencil is respectively $(y:z)\mapsto (-y:z)$, $(y:z)\mapsto (z:y)$ and $(y:z)\mapsto (y:z)$. The action of $G$ on the pencil thus gives an exact sequence
\begin{center}$1\rightarrow G'\rightarrow G\rightarrow (\Z{2})^2\rightarrow 0$,\end{center}
where $G'$ is the subgroup of elements of $G$ that act trivially on the pencil $\Lambda$. It is clear that $<h>\cong \Z{2n}$ is a subgroup of $G'$. Since $g_1 h(g_1)^{-1}=g_2 h(g_2)^{-1}=h^{-1}$ and $g_1$ and $g_2$ commute, the group $<h>$ is equal to $G'$.

Finally, any element of $G$ that fixes a curve of positive genus must act trivially on the pencil $\Lambda$ and thus belongs to $<h>$. Hence, only the identity is possible.
\end{proof}

\end{document}